\documentclass[reqno, 10pt]{amsart}
{
\usepackage{amsmath}
\usepackage[active]{srcltx}
\usepackage{amssymb}
\usepackage[dvips]{graphicx}
\usepackage{xcolor}
\usepackage{hyperref}
\usepackage[margin=1in]{geometry}

\usepackage{comment}

\usepackage{amsmath}
\usepackage[active]{srcltx}
\usepackage{amssymb}
\usepackage[dvips]{graphicx}
\usepackage{xcolor}
\usepackage{hyperref}
\usepackage{marginnote}

\usepackage{fancyhdr}

\fancyhf{}

\fancyhead[LE]{\small Weighted Uniform Distribution Along Primes} 
\fancyhead[RO]{\small V. Bergelson, G. Kolesnik and Y. Son}                 
\fancyfoot[C]{\thepage}                              

\pagestyle{fancy}
\fancypagestyle{plain}{%
  \fancyhf{} 
  \fancyfoot[C]{\thepage} 
}

\newtheorem{Theorem}{Theorem}[section]
\newtheorem{Lemma}[Theorem]{Lemma}
\newtheorem{Proposition}[Theorem]{Proposition}
\newtheorem{Corollary}[Theorem]{Corollary}
 \newtheorem*{theorem*}{Theorem}

\newtheorem{Example}[Theorem]{Example}

\theoremstyle{definition}
\newtheorem{Definition}[Theorem]{Definition}

\theoremstyle{definition}
\newtheorem{Remark}[Theorem]{Remark}
\numberwithin{equation}{section}
\usepackage{enumerate}

\begin{document}


\title{Weighted uniform distribution of subpolynomial functions along primes and applications} 
\author[V. Bergelson]{Vitaly Bergelson}
\address[V. Bergelson]{Department of Mathematics\\ Ohio State University \\ Columbus, OH 43210, USA}
\email{vitaly@math.ohio-state.edu}

\author[G. Kolesnik]{Grigori Kolesnik}
\address[G. Kolesnik]{Department of Mathematics\\ California State University \\ Los Angeles, CA 90032, USA}
\email{akolesnik628@gmail.com}

\author[Y. Son]{Younghwan Son}
\address[Y. Son]{Department of Mathematics\\ POSTECH \\ Pohang, 37673, South Korea}
\email{yhson@postech.ac.kr}

\date{\today}
\bigskip

\maketitle

\setcounter{section}{0}

\begin{abstract}
Let $u(x)$ be a subpolynomial function in a Hardy field. 
We establish necessary and sufficient conditions for the weighted uniform distribution of the sequences $(u(n))_{n\in\mathbb{N}}$ and $(u(p_n))_{n\in\mathbb{N}}$, where $p_n$ denotes the $n$-th prime.
This extends the main result of \cite{BKS} to the weighted setting and leads to new applications in uniform distribution theory, ergodic theory, and additive combinatorics. 
\end{abstract}

\renewcommand{\thefootnote}{\fnsymbol{footnote}} 
    \renewcommand{\thefootnote}{\arabic{footnote}} 

\section{Introduction}
The goal of this paper is to extend the main result obtained in \cite{BKS} (Theorem \ref{BKS:prev} below) to a wider class of functions coming from Hardy fields. 
This generalization enables us to obtain new results in the theory of uniform distribution, ergodic theory, and additive combinatorics. 

Let $\mathbf B$ denote the set of germs at $+ \infty$ of continuous real functions on $\mathbb{R}$.
Note that $\mathbf B$ forms a ring with respect to pointwise addition and multiplication.
A {\em Hardy field} is any subfield of $\mathbf B$ which is closed under differentiation. By ${\mathbf U}$ we denote the union of all Hardy fields.
A classical example of a Hardy field is the class $\mathbf{L}$ of logarithmico-exponential functions introduced by Hardy in \cite{Har1, Har2}, that is, the collection of all germs at $+ \infty$ of real-valued functions that can be constructed using the real constants, the functions $\exp x$ and $\log x$, and the operations of addition, multiplication, division and composition of functions. 

It is a classical fact that for any $u \in {\mathbf U}$, $\lim\limits_{x \rightarrow \infty} u(x)$ exists as an element of $\mathbb{R} \cup \{- \infty, \infty\}$. This implies that periodic functions such as $\sin x$ and $\cos x$ do not belong to ${\mathbf U}$. Also, if $u_1$ and $u_2$ belong to the same Hardy field, then the limit $\lim\limits_{x \rightarrow \infty} (u_1(x) - u_2(x))$ and the limit $\lim\limits_{x \rightarrow \infty} \frac{u_1(x)}{u_2(x)}$ exist (they may be infinite). 
In addition, $\mathbf{U}$ is closed under differentiation and integration.
See \cite{Bos} and the references therein for more information about Hardy fields.

Given two functions $u_1, u_2$ belonging to the same Hardy field, we write $u_1(x) \ll u_2(x)$ when there exists $C > 0$ such that $|u_1(x)| \leq C |u_2(x)|$ for all sufficiently large $x$, and $u_1(x) \prec u_2(x)$ when $\lim\limits_{x \rightarrow \infty} \frac{u_1(x)}{u_2(x)} = 0$.
A function $f \in \mathbf{U}$ is said to be subpolynomial if  $f(x) \ll x^n$ for some positive integer $n$.  

The following result was obtained in \cite{BKS}. 
(The main novelty in Theorem \ref{BKS:prev} is item (2); the equivalence of (1) and (3) was obtained by Boshernitzan in \cite{Bos}.)
Throughout this paper, we use the convention that $\mathbb{N} = \{1, 2, 3, \dots \}$. 

\begin{Theorem}[Theorem 1.6 in \cite{BKS}]
\label{BKS:prev}
For a subpolynomial function $u(x) \in \mathbf{U}$, the following are equivalent:
\begin{enumerate}
\item $(u(n))_{n \in \mathbb{N}}$ is uniformly distributed $\bmod \, 1$.
\item $(u(p_n))_{n \in \mathbb{N}}$ is uniformly distributed $\bmod \, 1$, where $p_n$ is the $n$-th prime. 
\item For any polynomial $q(x) \in \mathbb{Q}[x]$,
\begin{equation*} \lim_{x \rightarrow \infty} \frac{|u(x) - q(x) |}{\log x} = \infty.
\end{equation*}
\end{enumerate}
\end{Theorem}

Note that Theorem \ref{BKS:prev} implies the classical facts that $(\log n)_{n \in \mathbb{N}}$ and $(\log p_n)_{n \in \mathbb{N}}$ are not uniformly distributed $\bmod \, 1$. (See Example 2.4 on p.8 of \cite{KN}. See also \cite{Tsu} and \cite{Wi}.)
On the other hand, it is known (\cite{Tsu}) that $(\log n)_{n \in \mathbb{N}}$ is $1/n$-uniformly distributed $\bmod \, 1$, meaning that
for any interval $I = [a,b) \subset [0,1)$,
\begin{equation*} 
\lim_{N \rightarrow \infty} \frac{1}{\log N} \sum_{n=1}^N \frac{1}{n} 1_{[a,b)} (\{ \log n \}) = b-a.
 \end{equation*}
Here (and throughout the paper) $\{x\}$ denotes the fractional part of $x$. 

It is known (see \cite{Wh}) that the sequence $(\log p_n)_{n \in \mathbb{N}}$ is also $1/n$-uniformly distributed $\bmod \, 1$. 
The fact that $(\log n)_{n \in \mathbb{N}}$ and $(\log p_n)_{n \in \mathbb{N}}$ are $1/n$-uniformly distributed is a special case of a general theorem (Theorem \ref{BKS} below), which extends Theorem \ref{BKS:prev} to slowly growing functions from a Hardy field. 

In order to present our main result, we first need to introduce first the general definition of {\em weighted uniform distribution}.

\begin{Definition}
 \label{def:weighted-u.d.}
 Let $(w(n))_{n \in \mathbb{N}}$ be a non-increasing, positive sequence such that $\lim\limits_{N \rightarrow \infty} W(N) = \infty$, where $W(N) = \sum\limits_{n=1}^N w(n)$. 
A real sequence $(x_n)_{n \in \mathbb{N}}$ is {\em $w(n)$-uniformly distributed $\bmod \, 1$} if, for any interval $I = [a,b) \subset [0,1)$,
\begin{equation*}
\label{def:w-ud}
\lim_{N \rightarrow \infty} \frac{1}{ W(N)} \sum_{n=1}^N w(n) 1_{[a,b)} (\{ x_n \}) = b-a.
\end{equation*}
\end{Definition}

\begin{Remark}
In this paper, as a rule, the weight sequence $(w(n))_{n \in \mathbb{N}}$ will come from a Hardy field function $w(x)$ that has the following properties: 
\begin{enumerate}
\item $w(x)$ is positive and non-increasing,
\item $w(x) = W'(x)$ for some Hardy field function $W(x)$ such that $\lim\limits_{x \rightarrow \infty} W(x) = \infty$.
\end{enumerate}
In this case, we will use for the normalization the sequence $W(N), N \in \mathbb{N},$ rather than the sequence of discrete sums $\sum_{n=1}^N w(n)$. 
Clearly, such a replacement will not cause any trouble, since 
$$ \lim_{N \rightarrow \infty} \frac{\sum_{n=1}^N w(n)}{W(N)} = 1.$$ 
\end{Remark}

The examples discussed above demonstrate that the notion of $w(n)$-uniform distribution $\bmod \, 1$ is more general than that of the classical uniform distribution. 
Also, by applying summation by parts, one can show that any uniformly distributed $\bmod \, 1$ sequence is also $w(n)$-uniformly distributed $\bmod \, 1$:

\begin{Theorem} [cf. Lemma 7.1 of Chapter 1 in \cite{KN}; Theorem 7 in \cite{Tsu}]
\label{lem:Tsu1}
 Let $(w(n))_{n \in \mathbb{N}}$ be a non-increasing, positive sequence such that $\lim\limits_{N \rightarrow \infty} W(N) = \infty$, where $W(N) = \sum\limits_{n=1}^N w(n)$. 
 Let $(a_n)_{n \in \mathbb{N}}$ be a sequence of complex numbers. If 
 \[ \lim_{N \rightarrow \infty} \frac{1}{N} \sum_{n=1}^N a_n = a,\] 
 then 
  \[ \lim_{N \rightarrow \infty} \frac{1}{W(N)} \sum_{n=1}^N w(n) \, a_n = a.\] 
In particular, if a sequence $(x_n)_{n \in \mathbb{N}}$ of real numbers is uniformly distributed $\bmod \, 1$, then it is $w(n)$-uniformly distributed $\bmod \, 1$.
\end{Theorem}

The following result and its proof are analogous to those in the classical case.
\begin{Theorem}[Theorems 5 and 6 in \cite{Tsu}]
\label{lem:Tsu:Equ}
 Let $(w(n))_{n \in \mathbb{N}}$ be a non-increasing, positive sequence such that $\lim\limits_{N \rightarrow \infty} W(N) = \infty$, where $W(N) = \sum\limits_{n=1}^N w(n)$. 
The following statements are equivalent.
\begin{enumerate}
\item  A sequence $(x_n)_{n \in \mathbb{N}}$ is $w(n)$-uniformly distributed $\bmod \, 1$.
\item For any continuous function $f$ on $[0,1]$, 
\begin{equation}
\lim_{N \rightarrow \infty} \frac{1}{ W(N)} \sum_{n=1}^N w(n) f ( x_n ) = \int_0^1 f(t) \, dt.
\end{equation}
\item For any non-zero integer $m$,
\begin{equation}
\lim_{N \rightarrow \infty} \frac{1}{ W(N)} \sum_{n=1}^N w(n) e^{ 2 \pi i m  x_n } =0.
\end{equation}
\end{enumerate}
\end{Theorem}

If a Hardy field is not a proper subfield of a larger Hardy field, then it is called a maximal Hardy field. 
Let $\mathbf{E}$ be the intersection of all maximal Hardy fields. 
It is known that $\mathbf{E}$ is closed under integration and composition and contains the Hardy class $\mathbf{L}$. 
In fact, $\mathbf{E}$ is strictly larger than $\mathbf{L}$, as $\sin (1/x) \in \mathbf{E} \setminus \mathbf{L}$. Moreover, while $\mathbf{L}$ is not closed under integration, $\mathbf{E}$ is closed under integration. 
(See Corollary 8.3 in \cite{Bos1981} and Introduction of \cite{Bos1982}.)  
 
The following result is an extension of Theorem \ref{BKS:prev}. 
Note that $(1), (2)$, and $(3)$ in Theorem \ref{BKS:prev} correspond to $(1), (4),$ and $(7)$ in Theorem \ref{BKS} in the case when $W(x) = x$.
Notice also that the condition $W(x) \in \mathbf{E}$ guarantees that the limit in \eqref{eq:cond:wud} exists. 

\begin{Theorem}
\label{BKS}
Let $W(x) \in \mathbf{E}$ be such that $\lim\limits_{x \rightarrow \infty} W(x) = \infty$ and $w(x) := W'(x)$ is a non-increasing positive function.  
For a subpolynomial function $u(x) \in \mathbf{U}$, the following statements are equivalent:
\begin{enumerate}
\item $(u(n))_{n \in \mathbb{N}}$ is $w(n)$-uniformly distributed $\bmod \, 1$.
\item For any $a, d \in \mathbb{N}$, $(u(an + d))_{n \in \mathbb{N}}$ is $w(n)$-uniformly distributed $\bmod \, 1$.
\item For some $a, d \in \mathbb{N}$, $(u(an + d))_{n \in \mathbb{N}}$ is $w(n)$-uniformly distributed $\bmod \, 1$.
\item $(u(p_n))_{n \in \mathbb{N}}$ is $w(n)$-uniformly distributed $\bmod \, 1$, where $p_n$ is the $n$-th prime. 
\item For any $a, d \in \mathbb{N}$ with $\gcd (a,d)=1$, $(u(p_n^{a,d}))_{n \in \mathbb{N}}$ is $w(n)$-uniformly distributed $\bmod \, 1$, where $p_n^{a,d}$ is the $n$-th prime in the arithmetic progression $\{am + d: m = 0, 1, 2, \dots\}$. 
\item For some $a, d \in \mathbb{N}$ with $\gcd (a,d)=1$, $(u(p_n^{a,d}))_{n \in \mathbb{N}}$ is $w(n)$-uniformly distributed $\bmod \, 1$.
\item For any polynomial $q(x) \in \mathbb{Q}[x]$,
\begin{equation}
\label{eq:cond:wud}
\lim_{x \rightarrow \infty} \frac{|u(x) - q(x)|}{\log W(x)} = \infty.
\end{equation}
\end{enumerate}
\end{Theorem}

While Theorem \ref{BKS} is formulated for general subpolynomial functions in Hardy fields, in practice many natural examples arise from the class $\mathbf{E}$, the intersection of all maximal Hardy fields.  
In particular, for $u(x) \in \mathbf{E}$ with $1 \prec u(x) \prec x$, Theorem \ref{BKS} implies that both $(u(n))$ and $(u(p_n))$ are $u'(n)$-uniformly distributed $\bmod \, 1$.  
Here are some examples illustrating Theorem \ref{BKS}.

\begin{Example}
\label{ex:int}
\begin{enumerate}
\item For $0 < c \leq 1$,  both sequences $((\log n)^c)_{n \in \mathbb{N}}$ and $((\log p_n)^c)_{n \in \mathbb{N}}$ are  $1/n$-uniformly distributed $\bmod \, 1$. 
\item  More generally, for any $k \in \mathbb{N}$, define the $(k+1)$-fold iterated logarithm recursively by
$$\log^{(k+1)} x = \log \circ \log^{(k)} x = \log \log \cdots \log x \quad ((k+1)-\text{times}).$$ 
Let $w(x) = \frac{d}{dx} \log^{(k)} x = \frac{1}{ x \log x \log^{(2)} x \cdots \log^{(k-1)} x}$.
It is known (see \cite{Tsu}) that for $0 < c \leq 1$, $((\log^{(k)} n)^c)_{n \in \mathbb{N}}$ is $w(n)$-uniformly distributed $\bmod \, 1$. 
It follows from Theorem \ref{BKS} that
for $0 < c \leq 1$,  $((\log^{(k)} p_n)^c)_{n \in \mathbb{N}}$ also is $w(n)$-uniformly distributed $\bmod \, 1$.  
\item $(\log n + \log (n \log n))_{n \in \mathbb{N}}$ is $1/n$-uniformly distributed $\bmod \, 1$. 
It follows from Theorem \ref{lem2:sec2} below that  
$(\log n + \log p_n)_{n \in \mathbb{N}}$ is also $1/n$-uniformly distributed $\bmod \, 1$. 
\item $(\log n - \log (n \log n) )_{n \in \mathbb{N}}$ is not $1/n$-uniformly distributed $\bmod \, 1$, but $1/n \log n$-uniformly distributed $\bmod \, 1$. Similarly to (3), it follows from Theorem \ref{lem2:sec2} that $(\log n - \log p_n)_{n \in \mathbb{N}}$ is also  $1/n \log n$-uniformly distributed $\bmod \, 1$. 
\end{enumerate}
\end{Example}

In Subsection \ref{subsec:3.1}, we obtain a multidimensional version of Theorem \ref{BKS} (see Theorem \ref{ud:d-dim}), from which we derive in subsequent subsections diverse applications that deal with weighted ergodic averages, properties of sets of differences of sets of positive upper density in $\mathbb{Z}^k$, and new results on Benford's law.

We now formulate some results obtained in Section \ref{sec:app}. 
Given functions $u_1(x), \dots, u_k(x)$, we define 
\begin{equation}
\label{eq:def:span1}
\text{span}_{\mathbb{Z}}^* (u_1, \dots, u_k) := \left\{ \sum_{i=1}^k c_i u_i(x) : (c_1, \dots, c_k) \in \mathbb{Z}^k \setminus \{(0,0, \dots,0) \} \right\}
\end{equation}
and 
\begin{equation}
\label{eq:def:span2}
\text{span}_{\mathbb{R}}^* (u_1, \dots, u_k) := \left\{ \sum_{i=1}^k c_i u_i(x) : (c_1, \dots, c_k) \in \mathbb{R}^k \setminus \{(0,0, \dots,0) \} \right\}.
\end{equation}

The following theorem establishes convergence of weighted ergodic averages along functions from a Hardy field.  
\begin{Theorem}[Theorem \ref{thm:ergthm}]
Let $U_1, \dots, U_k$ be commuting unitary operators on a Hilbert space $\mathcal{H}$.
Let $P$ denote the orthogonal projection onto the invariant space $\mathcal{H}_{inv} = \{ f \in \mathcal{H} :  U_i f = f \,\, \textrm{for all} \,\, i = 1,2, \dots, k \}$. 
Let $W(x) \in {\mathbf E}$ be such that $\lim\limits_{x \rightarrow \infty} W(x) = \infty$ and $w(x) := W'(x)$ is positive, non-increasing.  
Let $u_1(x), \dots, u_k(x)$ be subpolynomial functions belonging to some Hardy field ${\mathbf H}$ and assume for any $u(x) \in \text{span}_{\mathbb{R}}^* (u_1(x), \dots, u_k(x))$,  
\[ \lim_{x \rightarrow \infty} \frac{|u (x) - q(x)|}{\log W(x)} = \infty \quad \text{for any } q(x) \in \mathbb{Q}[x].\]
Then, for any $f \in \mathcal{H}$,
\begin{equation}
\lim_{N \rightarrow \infty} \frac{1}{W(N)} \sum_{n=1}^N w(n) U_1^{\lfloor u_1(n) \rfloor} \cdots U_k^{ \lfloor u_k(n)\rfloor} f = Pf,
\end{equation}
and 
\begin{equation}
\lim_{N \rightarrow \infty} \frac{1}{W(N)} \sum_{n=1}^N w(n) U_1^{\lfloor u_1(p_n) \rfloor} \cdots U_k^{ \lfloor u_k(p_n) \rfloor} f = Pf.
\end{equation}
\end{Theorem}

Next, we formulate one of the results from Subsection \ref{subsec:setofrec} which deals with sets of differences of sets of positive upper Banach density in $\mathbb{Z}^k$. 
Recall that the upper Banach density of a set $E \subset \mathbb{Z}^k$ is defined by
$$d^*(E) = \sup_{\{\Pi_n\}_{n \in \mathbb{N}}} \limsup_{n \rightarrow \infty} \frac{|E \cap \Pi_n|}{| \Pi_n|},$$
where the supremum is taken over all sequences of parallelepipeds
$$ \Pi_n = [a_n^{(1)}, b_n^{(1)}] \times \cdots \times [a_n^{(k)}, b_n^{(k)} ] \subset \mathbb{Z}^k, \,\, n \in \mathbb{N},$$
with $b_n^{(j)} - a_n^{(j)} \rightarrow \infty$ for each $1 \leq j \leq k$.

\begin{Theorem}[Corollary \ref{cor:setofrec:H}]
Assume that $q_1(x), \dots, q_m(x) \in \mathbb{Z}[x]$ satisfy $q_i(0) = 0$ for all $i = 1, 2, \dots, m$, and let $u_1(x), \dots, u_k(x)$ be subpolynomial functions belonging to some Hardy field $\mathbf{H}$. 
Suppose that there exists a function $W(x) \in \mathbf{E}$ such that 
\begin{enumerate}[(i)]
\item $\lim\limits_{x \rightarrow \infty} W(x) = \infty$ and $w(x) := W'(x)$ is non-increasing and positive; 
\item for any $u(x) \in \text{span}_{\mathbb{R}}^* (u_1(x), \dots, u_k(x))$ and any $q(x) \in \mathbb{R}[x]$, 
\begin{equation}
\lim_{x \rightarrow \infty} \frac{|u(x) - q(x)|}{\log W(x)} = \infty.
\end{equation}
\end{enumerate}

Define
\begin{enumerate}
\item $D_1= \{  (q_1(n), \dots, q_m(n), \lfloor u_1(n) \rfloor, \dots, \lfloor u_k(n) \rfloor ) : n \in \mathbb{N} \}$,
\item $D_2 = \{  (q_1(p-1), \dots, q_m(p-1), \lfloor u_1(p) \rfloor, \dots, \lfloor u_k(p) \rfloor ) : p \in \mathcal{P} \}$,
\item $D_3 = \{  (q_1(p+1), \dots, q_m(p+1), \lfloor u_1(p) \rfloor, \dots, \lfloor u_k(p) \rfloor ) : p \in \mathcal{P} \}$,
\end{enumerate}
where  $\mathcal{P}$ denotes the set of prime numbers.
Then, for any set $E \subset \mathbb{Z}^{m+k}$ with $d^{*} (E) > 0$, $$(E - E) \cap D_i \ne \emptyset \quad  \text{for } i = 1, 2, 3.$$
\end{Theorem}

Finally, we present a result related to Benford's law.  
It is known (\cite{Dia}) that the sequence $(n!)_{n \in \mathbb{N}}$ obeys the first digit law: for any integer $a \in \{1, 2, \dots, 9\}$,
\[ \lim_{N \rightarrow \infty} \frac{1}{N} |\{1 \leq n \leq N: \text{ the first digit of } n! \text{ is } a \} | = \log_{10} (1+1/a). \]
Similarly, Mass\'{e} and Schneider \cite{MaSch} proved that the primorial sequence $(p_n \#)_{n \in \mathbb{N}}$, where $p_n\# = p_1 p_2 \dots p_n$ denotes the product of the first $n$ primes, also obeys the first digit law. 
The following result, obtained in Subsection \ref{subsec:3.4}, deals with the joint behavior of the first digits of $n!$ and $p_n\#$.

\begin{Theorem} [cf. Corollary \ref{thm:primorial:factorial}]
For any  $a, b \in \{1, 2, \dots, 9\}$, the logarithmic density of the set 
$$A_{a,b} = \{n \in \mathbb{N}:  n! \text{ starts with } a \text{ and } p_n\# \text{ starts with } b \}$$  is given by the product of the corresponding Benford probabilities:
\[ \lim_{N \rightarrow \infty} \frac{1}{\log N} \sum_{n=1}^N 1_{A_{a,b}} (n) = \log_{10} (1+ 1/a) \cdot \log_{10} (1+ 1/b).\]
\end{Theorem}

The structure of this paper is as follows. 
In Section 2, we collect various auxiliary results needed for the proof of the main theorem and its applications. 
In Section 3, we prove Theorem \ref{BKS}, the main result of this paper, and also establish several further results related to weighted uniform distribution. 
Section 4 is devoted to various applications, including the results formulated in this introduction.

\section*{Acknowledgement}
Younghwan Son is supported by Basic Science Research Institute Fund (2021R1A6A1A10042944) and the National Research Foundation of Korea (NRF) grant funded by the Korea government (MSIT) (RS-2025-24523340).

\section{Preliminaries}

In this section, we collect some useful results that will be used in the subsequent sections.
Throughout this section, unless otherwise stated, we assume that $(w(n))_{n \in \mathbb{N}}$ is a non-increasing sequence of positive real numbers, and that $W(N) := \sum\limits_{n=1}^N w(n)$ satisfies $\lim\limits_{N \rightarrow \infty} W(N) = \infty$. 

\subsection{Weighted uniform distribution}

In this subsection, we formulate two results concerning weighted uniform distribution $\bmod \, 1$. 
The first, Theorem \ref{lem:Tsu2} due to Tsuji (\cite{Tsu}), provides a sufficient condition for a sequence of the form $(u(n))_{n \in \mathbb{N}}$ to be $w(n)$-uniformly distributed $\bmod \, 1$. 
The second, Theorem \ref{lem:Nie}, is a special case of a theorem due to Niederreiter (\cite{N}) and highlights a necessary growth condition for monotone sequences that are $w(n)$-uniformly distributed $\bmod \,1$. 

\begin{Theorem} [Theorem 8 in \cite{Tsu}]
\label{lem:Tsu2}
Let $W(x)$ be a real-valued function on $[1, \infty)$. 
Assume that $W(x)$ is twice continuously differentiable, satisfies $\lim\limits_{x \rightarrow \infty} W(x) = \infty$ and that its derivative $w(x) := W'(x)$ is positive and non-increasing.
Let $u(x)$ be a positive continuous increasing function with a continuous derivative $u'(x)$ for $x \geq 1$, and assume the following conditions are satisfied:
\begin{enumerate}[(i)]
\item $\lim\limits_{x \rightarrow \infty} u(x) = \infty$,
\item $u'(x) \rightarrow 0$ monotonically as $x \rightarrow \infty$,
\item $u'(x)/w (x)$ is eventually monotone,
\item $\lim\limits_{x \rightarrow \infty} \frac{u'(x)}{w (x)} W(x) = \infty$.
\end{enumerate} 
Then $(u(n))_{n \in \mathbb{N}}$ is $w(n)$-uniformly distributed $\bmod \, 1$.
\end{Theorem}
It is worth noting that Theorem \ref{lem:Tsu2} is a generalization of the classical Fej\'er theorem (\cite{KN}, Section 1.2, Corollary 2.1), which corresponds to the case $w(x) = 1$.

\begin{Theorem}[cf.  Theorem 2 in \cite{N}]
\label{lem:Nie}
If $(x_n)_{n \in \mathbb{N}}$ is a monotone $w(n)$-uniformly distributed $\bmod \, 1$ sequence, then
\[ \limsup_{n \rightarrow \infty} \frac{|x_n|}{\log W(n)} = \infty.\] 
\end{Theorem}

\subsection{Behavior of slowly growing Hardy field functions along the prime numbers}
In this subsection, we present two useful results (Theorem \ref{lem2:sec2} and Theorem \ref{lem3}) that will be utilized in the proofs of various averaging results in Sections \ref{sec2} and \ref{sec:app} for sequences of the form $u(p_n)$, where $u(x)$ is a slowly growing function in the sense that $\lim\limits_{x \rightarrow \infty} u(x) = \infty$ and  $\lim\limits_{x \rightarrow \infty} \frac{u(x)}{\log x} < \infty$.

The proof of Theorem \ref{lem2:sec2} relies on the following classical result. We include the proof for the reader's convenience.

\begin{Theorem}
\label{lem:PNT}
For $a, d \in \mathbb{N}$ with $\gcd(a,d) =1$, let $(p_n^{a,d})_{n \in \mathbb{N}}$ be the $n$-th prime in the arithmetic progression $\{ am +d: m = 0, 1, 2, \dots \}$. 
Then
\begin{equation*} 
\lim_{n \rightarrow \infty} \frac{p_n^{a,d}}{ \phi(a) n \log n} =1,
\end{equation*}
where $\phi$ is the Euler totient function.
\end{Theorem}

\begin{proof}
Let $$\pi(x) := \sum_{p \leq x} 1 \quad \text{ and } \quad \pi_{a,d}(x) = \sum_{\substack{ p \leq x \\ p \equiv d  \bmod a} } 1.$$
By the prime number theorem and its generalization to arithmetic progressions (Theorems 4.5 and 7.10 in \cite{Ap}), we have
\begin{equation}
\label{eq:PNT} 
\lim_{x \rightarrow \infty} \frac{\pi (x)}{ x / \log x } =1, \quad \lim_{x \rightarrow \infty} \frac{\pi_{a,d} (x)}{ x / \phi(a) \log x } =1.
\end{equation}
It follows from \eqref{eq:PNT} that 
\begin{equation}
\label{eq:cons:pnt}
 \lim_{x \rightarrow \infty} \frac{\pi_{a,d} (x) \log \pi_{a,d} (x) }{x} = \lim_{x \rightarrow \infty} \frac{\pi(x) \log \pi(x)}{\phi(a) x} = \frac{1}{\phi(a)}.
 \end{equation}
If $x = p_{n}^{a,d}$, then $\pi_{a,d} (x) = n$. Substituting this into \eqref{eq:cons:pnt}, we get
\begin{equation*} 
\lim_{n \rightarrow \infty} \frac{p_n^{a,d}}{ \phi(a) n \log n} =1. 
\end{equation*}
\end{proof}

\begin{Theorem}
\label{lem2:sec2}
Let $u(x) \in {\mathbf U}$ be such that $\lim\limits_{x \rightarrow \infty} |u(x)| = \infty$ and $\lim\limits_{x \rightarrow \infty} \frac{|u(x)|}{\log x} < \infty$. 
For $a, d \in \mathbb{N}$ with $\gcd(a,d) =1$, let $(p_n^{a,d})_{n \in \mathbb{N}}$ be the $n$-th prime in the arithmetic progression $\{ am +d: m = 0, 1, 2, \dots \}$. 
Let $(x_n)_{n \in \mathbb{N}}$ be a sequence of real numbers. 
Then,
 ${(x_n + u(p_n^{a,d}))_{n \in \mathbb{N}}}$ is $w(n)$-uniformly distributed $\bmod \, 1$ if and only if $ (x_n + u( \phi(a) n \log n))_{n \in \mathbb{N}}$ is $w(n)$-uniformly distributed $\bmod \, 1$, where $\phi$ is the Euler totient function.
\end{Theorem}

\begin{proof}

By L'Hospital's rule, 
\[ \lim_{x \rightarrow \infty} x u'(x) = \lim_{x \rightarrow \infty}  \frac{u(x)}{\log x}.\]
Since $\lim_{x \rightarrow \infty} \frac{|u(x)|}{\log x} < \infty$, we have that $\lim_{x \rightarrow \infty} x |u'(x)| < \infty$, and so there exists $x_0$ and $A> 0$ such that for all $t \geq x_0$, $|u'(t)| \leq \frac{A}{t}$. 

For any real numbers $x$ and $y$ with $x \geq 2 x_0$ and $|x-y| \leq x/2$, we have, by the Mean Value Theorem,
\begin{equation}
\label{eq:meanvalue} 
 |u(x) - u(y)| = \sup_{t \in [x/2, 3x/2]} |u'(t)| \cdot |x-y| \leq \frac{ 2A}{x} |x-y|. 
 \end{equation}
In view of Theorem \ref{lem:Tsu:Equ}, in order to complete the proof, it suffices to show that for any non-zero integer $m$,
\begin{equation} 
\label{eq2.4}
\lim_{N \rightarrow \infty} \left|  \frac{1}{W(N)} \sum_{n=1}^N w(n) e^{ 2 \pi i m (x_n+ u(p_n^{a,d}))}  -  \frac{1}{W(N)} \sum_{n=1}^N w(n) e^{2 \pi i m (x_n+ u( \phi(a) n \log n)} \right| =0. 
\end{equation}

By \eqref{eq:meanvalue} with $x= \phi(a) n\log n$ and $y = p_n^{a,d}$ and Theorem \ref{lem:PNT}, we obtain
\begin{equation}
\label{eqn:comp:prime-log}
 |u(p_n^{a,d}) -  u(\phi(a) n \log n)| \leq 2A \left( \frac{|p_n^{a,d} - \phi(a) n \log n|}{\phi(a) n \log n } \right) \rightarrow 0 \quad \text{as } n \rightarrow \infty.
 \end{equation}
Thus,
\begin{equation*}
 \left| e^{ 2 \pi i m (x_n + u(p_n^{a,d}))} - e^{ 2 \pi i m (x_n + u(\phi(a) n \log n))} \right| 
 \leq 2 \pi m \left( \left|u(p_n^{a,d}) - u(\phi(a) n \log n) \right| \right) \rightarrow 0 \quad \text{as } n \rightarrow \infty. 
 \end{equation*} 
This establishes \eqref{eq2.4}, completing the proof.
\end{proof}

Putting $x_n =0$ in the formulation of Theorem \ref{lem2:sec2} gives the following useful corollary.

\begin{Corollary}
\label{cor:prime:replace}
Let $u(x) \in {\mathbf U}$ be such that $\lim\limits_{x \rightarrow \infty} |u(x)| = \infty$ and $\lim\limits_{x \rightarrow \infty} \frac{|u(x)|}{\log x} < \infty$. 
Then ${( u(p_n^{a,d}))_{n \in \mathbb{N}}}$ is $w(n)$-uniformly distributed $\bmod \, 1$ if and only if $ ( u( \phi(a) n \log n))_{n \in \mathbb{N}}$ is $w(n)$-uniformly distributed $\bmod \, 1$.
In particular, ${( u(p_n))_{n \in \mathbb{N}}}$ is $w(n)$-uniformly distributed $\bmod \, 1$ if and only if $ ( u( n \log n))_{n \in \mathbb{N}}$ is $w(n)$-uniformly distributed $\bmod \, 1$.
\end{Corollary}

We conclude this section by showing that if $u(x) \in \mathbf{U}$ satisfies $\lim\limits_{x \rightarrow \infty} u(x) = \infty$ and $\lim\limits_{x \rightarrow \infty} \frac{u(x)}{\log x} < \infty$, then $u (x \log x)$ has the same asymptotic growth rate as $u(x)$, so 
\begin{equation} 
\lim_{x \rightarrow \infty} \frac{u( x \log x)}{\log W(x)}  = \lim_{x \rightarrow \infty} \frac{u(x)}{\log W(x)}. 
\end{equation}
This fact will be utilized in the proof of the equivalence of conditions (1) and (2) in Theorem \ref{BKS} (see Theorem \ref{BKS2} in Section \ref{sec2}).

Recall that if $u(x) \in {\mathbf U}$, then both $u(ax+d)$ ($a, d \in \mathbb{R}$ with $a>0$) and $u(x \log x)$ belong to ${\mathbf U}$ as well. This follows from the following two observations:
\begin{itemize}
\item the compositional inverses of $ax+d$ and $x \log x$ belong to $\mathbf{E}$ (Theorem 8.4 in \cite{Bos1981}), and 
\item if $f \in \mathbf{E}$ with $\lim\limits_{x \rightarrow \infty} f(x) = \infty$ and $g \in \mathbf{U}$, then $g \circ f^{-1}(x) \in \mathbf{U}$ (Lemma 2.4 in \cite{Bos}).
\end{itemize}

\begin{Theorem}
\label{lem3}
Let $u(x) \in {\mathbf U}$ be such that $\lim\limits_{x \rightarrow \infty} u(x) = \infty$ and $\lim\limits_{x \rightarrow \infty} \frac{u(x)}{\log x} < \infty$. 
Then 
\begin{equation}
\label{eq1:Lem2.8}
\lim_{x \rightarrow \infty} \frac{u(x \log x)}{u(x)} = 1.
\end{equation}
In addition, for any integers $a, d$ with $a \geq 1$, 
\begin{equation}
\label{eq2:Lem2.8}
\lim_{x \rightarrow \infty} \frac{u(ax +d)}{u(x)} = 1 
\end{equation}
and 
\begin{equation}
\label{eq3:Lem2.8}
 \lim_{x \rightarrow \infty} \frac{u(a x \log x)}{u(x)} = 1.
\end{equation}
Therefore, 
\begin{equation} 
\lim_{x \rightarrow \infty} \frac{u(a x \log x)}{\log W(x)} = \lim_{x \rightarrow \infty} \frac{u(a x + d)}{\log W(x)}  = \lim_{x \rightarrow \infty} \frac{u(x)}{\log W(x)}. 
\end{equation}
\end{Theorem}
\begin{proof}
We first prove \eqref{eq1:Lem2.8}. 
We will show that 
\begin{equation}\label{eq:Lem2.8:goal}
 \lim_{x \rightarrow \infty} \frac{u(x \log x) - u(x)}{u(x)} = 0.
 \end{equation}
Let us consider two cases: 
\[ (i) \lim_{x \rightarrow \infty} \frac{u(x)}{ \log \log x} < \infty  \quad \text{and} \quad (ii)   \lim_{x \rightarrow \infty} \frac{u(x)}{ \log \log x} = \infty. \]
Case $(i)$. 
Applying L'Hospital's rule, we get that case (i) is equivalent to 
\begin{equation}
\label{eq:lem2.8:case(i)}
\lim_{x \rightarrow \infty} (x \log x) u'(x)  < \infty. 
\end{equation}
Since $u'(x) {\in \mathbf U}$, \eqref{eq:lem2.8:case(i)} implies that there exists $x_0$ such that for $x \geq x_0$, $u'(x)$ is decreasing.
Thus, for $x \geq x_0$,
\begin{equation}
\label{eq2.9}
\frac{u(x \log x) - u(x)}{u(x)} = \frac{\int_x^{x \log x} u'(t) \, dt}{u(x)} \leq \frac{(x \log x - x) u'(x)}{u(x)} \leq \frac{(x \log x) u'(x)}{u(x)}.
\end{equation}
Now notice that \eqref{eq:lem2.8:case(i)} implies that the right hand side (and hence the left hand side) in formula \eqref{eq2.9} converges to $0$ as $x \rightarrow \infty$,  which proves formula  \eqref{eq:Lem2.8:goal}.

Case $(ii)$. Applying L'Hospital's rule, we see that  $\lim\limits_{x \rightarrow \infty} \frac{u(x)}{\log x} < \infty$ implies
$$\lim\limits_{x \rightarrow \infty} x u'(x) = A < \infty.$$ 
Thus, for some $C$ with $C \geq A +1$, there exist $x_0$ such that for $x \geq x_0$,  $u'(x) \leq \frac{C}{x}$.
For $x \geq x_0$,
\begin{align} 
\label{eq2.12}
\frac{u(x \log x) - u(x)}{u(x)} &=   \frac{\int_x^{x \log x} u'(t) \, dt }{u(x)} 
\leq  \frac{\int_x^{x \log x} C/t \, dt }{u(x)} \notag \\
&= \frac{C \log (x\log x) - C \log x}{u(x)} =  C \frac{  \log \log x}{u(x)}.
\end{align}
Formula  \eqref{eq:Lem2.8:goal} follows by taking $x \rightarrow \infty$ in \eqref{eq2.12}.

Now it remains to prove \eqref{eq2:Lem2.8} and \eqref{eq3:Lem2.8}.
Since $u(x)$ is eventually increasing, for sufficiently large $x$ we have 
$$u(x) \leq u (ax+d) \leq u(x \log x),$$ 
which can be rewritten as
\begin{equation} 
\label{eq:2.13:lem2.6}
1 \leq \frac{u (ax+d)}{u(x)} \leq \frac{u(x \log x)}{u(x)}.
\end{equation}
Now formulas \eqref{eq1:Lem2.8} and \eqref{eq:2.13:lem2.6} give \eqref{eq2:Lem2.8}:
\begin{equation}
\label{eq:2.14}
\lim_{x \rightarrow \infty} \frac{ u(ax+d)}{u(x)} =1.
\end{equation}
Applying \eqref{eq1:Lem2.8} and \eqref{eq:2.14}, we further obtain  \eqref{eq3:Lem2.8}: 
\begin{equation}
\label{eq:2.15}
\lim_{x \rightarrow \infty} \frac{ u(ax \log x)}{u(x)} = \lim_{x \rightarrow \infty} \frac{ u(ax \log x)}{u(x \log x)} \frac{ u(x \log x)}{u(x)} =1.
\end{equation}
\end{proof}

\subsection{Preservation of weighted uniform distribution under perturbations by slowly growing sequences}
The following result will be used in Subsection \ref{subsec:3.4}, where we obtain some results pertaining to Benford's law. 

\begin{Theorem}
\label{thm:slowper}
Let $W(x) \in \mathbf{E}$ be such that $\lim\limits_{x \rightarrow \infty} W(x) = \infty$ and $w(x) := W'(x)$ is a non-increasing positive function.
Assume that $u(x) \in \mathbf{U}$ such that 
\begin{equation} 
\label{cond:Cor:slowper}
\lim\limits_{x \rightarrow \infty} \frac{|u(x)|}{\log W(x)} < \infty.
\end{equation}
Then, if $(x_n)_{n \in \mathbb{N}}$ is $w(n)$-uniformly distributed $\bmod \, 1$, then so are both $(x_n + u(n))_{n \in \mathbb{N}}$ and $(x_n + u(p_n))_{n \in \mathbb{N}}$.  
In particular, if $(x_n)_{n \in \mathbb{N}}$ is uniformly distributed $\bmod \, 1$, then so are $(x_n + \log n)_{n \in \mathbb{N}}$ and $(x_n + \log p_n)_{n \in \mathbb{N}}$. 
\end{Theorem}

\begin{Remark}
By Theorem \ref{BKS}, condition \eqref{cond:Cor:slowper} implies that neither $(u(n))_{n \in \mathbb{N}}$ nor $(u(p_n))_{n \in \mathbb{N}}$ is $w(n)$-uniformly distributed $\bmod \, 1$.
\end{Remark}

To prove Theorem \ref{thm:slowper}, we need the following lemma.
\begin{Lemma}
\label{lem:slowper}
Suppose that $(y_n)_{n \in \mathbb{N}}$ is a sequence of real numbers such that there exists $C > 0$ for which
\begin{equation}
|y_{n} - y_{n-1}| \leq C \frac{w(n)}{W(n)} \quad \text{for all } n \geq 2.
\end{equation}
If $(x_n)_{n \in \mathbb{N}}$ is $w(n)$-uniformly distributed $\bmod \, 1$, then the sequence $(x_n + y_n)_{n \in \mathbb{N}}$ is also $w(n)$-uniformly distributed $\bmod \, 1$.  
In particular, if $(x_n)_{n \in \mathbb{N}}$ is uniformly distributed $\bmod \, 1$ and $|y_n - y_{n-1}| \leq \frac{C}{n}$ for all $n$, then the sequence $(x_n + y_n)_{n \in \mathbb{N}}$ is uniformly distributed $\bmod \, 1$.  
\end{Lemma}

\begin{proof}
In light of Theorem \ref{lem:Tsu:Equ}, it suffices to show that for any non-zero integer $m$, 
\begin{equation}
\label{eq:thm:slowper}
\lim_{N \rightarrow \infty} \frac{1}{W(N)} \sum_{n=1}^N w(n) e^{2 \pi i m(x_n + y_n)} = 0.
\end{equation}
Summation by parts gives
\begin{align*}
\sum_{n=1}^N w(n) e^{2 \pi i m(x_n + y_n)} &= \sum_{n=1}^N e^{2 \pi i m y_n} w(n) e^{2 \pi m x_n} \\
&= \sum_{n=1}^{N-1} (e^{2 \pi i m y_n} - e^{2 \pi i m y_{n-1}}) \sum_{l=1}^n w(l) e^{ 2 \pi i m x_l} + e^{2 \pi i m y_N} \sum_{l=1}^N w(l) e^{2 \pi i m x_l}.
\end{align*}
Since $|e^{ 2 \pi i my_n} - e^{2 \pi i m y_{n-1}}| \leq 2 \pi m  |y_n - y_{n-1}| \leq 2 \pi mC \frac{w(n)}{W(n)}$, 
\[ \frac{1}{W(N)} \left| \sum_{n=1}^N w(n) e^{2 \pi i m(x_n + y_n)} \right| \leq \frac{2 \pi m C }{W(N)} \sum_{n=1}^{N-1} w(n) \left| \frac{1}{W(n)} \sum_{l=1}^n w(l) e^{2 \pi i m x_l} \right| + \frac{1}{W(N)} \left|  \sum_{l=1}^N w(l) e^{2 \pi i m x_l} \right|. \]
Now \eqref{eq:thm:slowper} follows from the fact that $\lim\limits_{n \rightarrow \infty} \frac{1}{W(n)}\left| \sum\limits_{l=1}^n w(l) e^{2 \pi i m x_l} \right| = 0$.   
\end{proof}

\begin{proof}[Proof of Theorem \ref{thm:slowper}]
Let us first show that $(x_n + u(n))_{n \in \mathbb{N}}$ is $w(n)$-uniformly distributed 
$\bmod \, 1$.
By L'Hospital's rule and Theorem \ref{lem3}, $$\lim_{x \rightarrow \infty} \frac{w(x)/ W(x)}{w(x+1) /W(x+1)} = \lim_{x \rightarrow \infty} \frac{\log W(x)}{\log W(x+1)} =1.$$
Thus, 
 $$\lim\limits_{x \rightarrow \infty} u'(x) \frac{W(x+1)}{ w(x+1)} =  \lim\limits_{x \rightarrow \infty} u'(x) \frac{W(x)}{ w(x)} = \lim\limits_{x \rightarrow \infty} \frac{u(x)}{\log W(x)} < \infty.$$ 
It follows that there exists $C > 0$ such that $|u'(x)| \leq C \frac{w(x+1)}{W(x+1)}$. 
Moreover, as $\lim\limits_{x \rightarrow \infty} \frac{w(x)}{W(x)} =0$, $|u'(x)|$ is eventually decreasing. 
This, in turn, implies
\[ |u(n) - u(n-1)| \leq |u'(n-1)| \leq C \frac{w(n)}{W(n)}. \] 
Therefore, by Lemma \ref{lem:slowper}, $(x_n + u(n))_{n \in \mathbb{N}}$ is $w(n)$-uniformly distributed $\bmod \, 1$.
We also note that  
\[
\lim_{x \rightarrow \infty} \frac{|u(x \log x)|}{\log W(x)} < \infty.
\]  
This is immediate if \( \lim\limits_{x \rightarrow \infty} u(x) < \infty \), and it follows from Theorem \ref{lem3} when \( \lim\limits_{x \rightarrow \infty} u(x) = \infty \).  
Consequently, the sequence \( (x_n + u(n \log n))_{n \in \mathbb{N}} \) is also \( w(n) \)-uniformly distributed $\bmod \, 1$.  
Finally, note that the condition \( \lim\limits_{x \rightarrow \infty} \frac{u(x)}{\log W(x)} < \infty \) implies  
\[
\lim_{x \rightarrow \infty} \frac{u(x)}{\log x} < \infty.
\]  
This observation allows us to apply Theorem \ref{lem2:sec2}, from which we conclude that the sequence \( (x_n + u(p_n))_{n \in \mathbb{N}} \) is also \( w(n) \)-uniformly distributed $\bmod \, 1$.
\end{proof}

\subsection{A weighted difference theorem and an estimate for exponential sums for slowly growing functions}
In this subsection, we collect some auxiliary results that will be utilized in Subsection \ref{subsec:3.4}. 
The first result formulated in this subsection, Theorem \ref{vdC}, is a special case of Lemma 4.6 in \cite{BMR} and is given without proof.
We also present a corollary of Theorem \ref{vdC} (Corollary \ref{Cor:vdCdiff}), which can be viewed as a weighted version of the classical van der Corput difference theorem. 
We conclude this subsection with Theorem \ref{lem:estimate:expsum:Eulersum}, which provides a useful auxiliary estimate for exponential sums involving slowly growing functions.

\begin{Theorem}[cf. Lemma 4.6 in \cite{BMR}]
\label{vdC}
Let $(x_n)_{n \in \mathbb{N}}$ be a sequence of real numbers. 
Suppose that 
\begin{equation} \lim_{h \rightarrow \infty} \limsup_{N \rightarrow \infty} \left| \frac{1}{W(N)} \sum_{n=1}^N w(n) \, e^{2 \pi i (x_{n+h} - x_n)}  \right| = 0.\end{equation}
Then 
\begin{equation} \lim_{N \rightarrow \infty} \frac{1}{W(N)} \sum_{n=1}^N  w(n) \, e^{2 \pi i x_n} = 0.\end{equation}
\end{Theorem}

The following corollary of Theorem \ref{vdC} is a natural weighted version of the classical van der Corput difference theorem, which corresponds to the case $w(n)=1$, $n\in\mathbb{N}$.

\begin{Corollary}
\label{Cor:vdCdiff}
Let $(x_n)_{n \in \mathbb{N}}$ be a sequence of real numbers. 
If for every positive integer $h$, $x_{n+h} - x_n, n = 1, 2, \dots,$ is $w(n)$-uniformly distributed $\bmod \, 1$, then so is $(x_n)_{n \in \mathbb{N}}$.
\end{Corollary}
\begin{proof}
Assume that for every positive integer $h$, $x_{n+h} - x_n, n = 1, 2, \dots,$ is $w(n)$-uniformly distributed $\bmod \, 1$. 
Then applying Theorem \ref{lem:Tsu:Equ} (3), we have for every nonzero integer $m$, 
\begin{equation} 
 \lim_{N \rightarrow \infty}  \frac{1}{W(N)} \sum_{n=1}^N w(n) \, e^{2 \pi i m(x_{n+h} - x_n)}  = 0.
\end{equation}
Applying Theorem \ref{vdC}, we obtain
\begin{equation} 
\lim_{N \rightarrow \infty} \frac{1}{W(N)} \sum_{n=1}^N  w(n) \, e^{2 \pi i m x_n} = 0 \quad \text{for all } m \in \mathbb{Z} \setminus \{ 0 \}.
\end{equation}
Invoking Theorem \ref{lem:Tsu:Equ} again, we obtain that $(x_n)_{n \in \mathbb{N}}$  is $w(n)$-uniformly distributed $\bmod \, 1$.
\end{proof}

\begin{Theorem}
\label{lem:estimate:expsum:Eulersum}
Let $W(x) \in \mathbf{E}$ be such that $\lim\limits_{x \rightarrow \infty} W(x) = \infty$ and $w(x) := W'(x)$ is a non-increasing positive function.  
Let $u(x) \in \mathbf{U}$. Suppose that there exists a nonzero real number $a$ such that
\[ \lim_{x \rightarrow \infty} \frac{u(x)}{\log W(x)} = a.\]
Then,
\[ \lim_{|h| \rightarrow \infty} \limsup_{N \rightarrow \infty} \left| \frac{1}{W(N)} \sum_{n=1}^N w(n) e^{2 \pi i h u(n)}  \right| =0.\] 
\end{Theorem}

\begin{proof}
Let $p(x) = x - \lfloor x \rfloor + \frac{1}{2}$. 
Then, by the Euler summation formula, we have that  
\begin{equation}
\label{eq:Eulersum:sec2}
 \sum_{n=1}^N w(n) e^{2 \pi i h u(n)} = \int_1^N w(x) e^{2 \pi i h u(x)} \, dx + \frac{1}{2} \left(w(N) e^{2 \pi i h u(N)} + w(1) e^{2 \pi i h u(1)} \right) - \int_1^N p(x) \frac{d}{dx} (w(x) e^{2 \pi i h u(x)}) \, dx.
 \end{equation}
Notice that the middle term on the right-hand side of \eqref{eq:Eulersum:sec2}, when divided by $W(N)$, tends to zero. Note also that 
\[ \left| \int_1^N p(x) \frac{d}{dx} (w(x) e^{2 \pi i h u(x)}) \, dx \right| \leq \int_1^N (- w'(x) + u'(x)) \, dx = w(1) - w(N) + u(N) - u(1). \]
Since, $\lim\limits_{N \rightarrow \infty} \frac{u(N)}{\log W(N)} = a$, the last term on the right-hand side of \eqref{eq:Eulersum:sec2}, when divided by $W(N)$, also tends to zero.
Therefore, we have that 
\begin{equation}
\label{eq2.27:conv:sumtoint}
 \lim_{N \rightarrow \infty} \left| \frac{1}{W(N)} \sum_{n=1}^N w(n) e^{2 \pi i h r(n)} - \frac{1}{W(N)} \int_1^N w(x) e^{2 \pi i h r(x)} \, dx \right| =0 
\end{equation}
To deal with the term $\frac{1}{W(N)} \int_1^N w(x) e^{2 \pi i h r(x)} \, dx$, we proceed as follows:
\begin{align*} 
 \int_1^N w(x) e^{2 \pi i h u(x)} \, dx  &=  \left[  W(x) e^{2 \pi i h u(x)} \right]_1^N - 2 \pi i h \int_1^N W(x) u'(x) e^{2 \pi i h u(x)} dx \\
 & =  \left[  W(x) e^{2 \pi i h u(x)} \right]_1^N - 2 \pi i h \int_1^N \left(\frac{W(x) u'(x)}{w(x)} - a + a \right) w(x) e^{2 \pi i h u(x)} dx  \\
 &=  \left[  W(x) e^{2 \pi i h u(x)} \right]_1^N - 2 \pi i h \int_1^N \left(\frac{W(x) u'(x)}{w(x)} - a \right) w(x) e^{2 \pi i h u(x)} dx \\
 &\quad -  2 \pi i h a \int_1^N w(x) e^{2 \pi i h u(x)} dx
\end{align*}
Rearranging the terms and dividing by $W(N)$, we obtain
\begin{equation}
\label{eq2.10:sec2}
\frac{1+ 2 \pi i h a}{W(N)} \int_1^N w(x) e^{2 \pi i h u(x)} \, dx = \frac{1}{W(N)}  \left[  W(x) e^{2 \pi i h u(x)} \right]_1^N 
- \frac{2 \pi i h}{W(N)} \int_1^N \left(\frac{W(x) u'(x)}{w(x)} - a \right) w(x) e^{2 \pi i h u(x)} dx 
\end{equation}
Note that 
\begin{equation}
\label{eq1:est1:thm2.15}
\limsup_{N \rightarrow \infty}  \left| \frac{1}{W(N)}  \left[  W(x) e^{2 \pi i h u(x)} \right]_1^N \right| =1. 
\end{equation}
In addition, since
\[
\lim_{x \rightarrow \infty} \frac{W(x) u'(x)}{w(x)} = \lim_{x \rightarrow \infty} \frac{u(x)}{\log W(x)} = a,
\]
we have
\begin{equation} 
\label{eq2:est2:thm2.15}
\lim_{N \rightarrow \infty}  \frac{1}{W(N)} \int_1^N \left(\frac{W(x) u'(x)}{w(x)} - a \right) w(x) e^{2 \pi i h u(x)} dx =0. 
\end{equation}
Dividing \eqref{eq2.10:sec2} by $1+2 \pi i ha$ and substituting \eqref{eq1:est1:thm2.15} and \eqref{eq2:est2:thm2.15} into the two terms of the right hand side of \eqref{eq2.10:sec2}, we obtain
\[  \limsup_{N \rightarrow \infty} \left| \frac{1}{W(N)} \int_{1}^N w(x) e^{2 \pi i h u(x)} \, dx  \right| = \frac{1}{|1+ 2 \pi i ha|}.  \]
Then, using \eqref{eq2.27:conv:sumtoint},  we have
\[  \limsup_{N \rightarrow \infty} \left| \frac{1}{W(N)} \sum_{n=1}^N w(n) e^{2 \pi i h u(n)}  \right| =  \limsup_{N \rightarrow \infty} \left| \frac{1}{W(N)} \int_{1}^N w(x) e^{2 \pi i h u(x)} \, dx  \right| = \frac{1}{|1+ 2 \pi i ha|}.  \]
As the right-hand side tends to zero as $|h| \rightarrow \infty$, the conclusion follows.
\end{proof}

\section{Weighted uniform distribution of subpolynomials in Hardy fields}
\label{sec2}
In this section, we prove our main theorem, Theorem \ref{BKS}, and then derive several further results concerning weighted uniform distribution.
Throughout this section, we assume that $W(x) \in \mathbf{E}$, $\lim\limits_{x \rightarrow \infty} W(x) = \infty$, and that $w(x) := W'(x)$ is a non-increasing positive function.

\begin{Theorem}[Theorem \ref{BKS} in Introduction]
\label{BKS2}
For a subpolynomial function $u(x) \in \mathbf{U}$, the following are equivalent:
\begin{enumerate}
\item $(u(n))_{n \in \mathbb{N}}$ is $w(n)$-uniformly distributed $\bmod \, 1$.
\item For any $a, d \in \mathbb{N}$, $(u(an + d))_{n \in \mathbb{N}}$ is $w(n)$-uniformly distributed $\bmod \, 1$.
\item For some $a, d \in \mathbb{N}$, $(u(an + d))_{n \in \mathbb{N}}$ is $w(n)$-uniformly distributed $\bmod \, 1$.
\item $(u(p_n))_{n \in \mathbb{N}}$ is $w(n)$-uniformly distributed $\bmod \, 1$, where $p_n$ is the $n$-th prime. 
\item For any $a, d \in \mathbb{N}$ with $\gcd (a,d)=1$, $(u(p_n^{a,d}))_{n \in \mathbb{N}}$ is $w(n)$-uniformly distributed $\bmod \, 1$, where $p_n^{a,d}$ is the $n$-th prime in the arithmetic progression $\{am + d: m = 0, 1, 2, \dots\}$. 
\item For some $a, d \in \mathbb{N}$ with $\gcd (a,d)=1$, $(u(p_n^{a,d}))_{n \in \mathbb{N}}$ is $w(n)$-uniformly distributed $\bmod \, 1$.
\item For any polynomial $q(x) \in \mathbb{Q}[x]$,
\[ \lim_{x \rightarrow \infty} \frac{|u(x) - q(x)|}{\log W(x)} = \infty.\]
\end{enumerate}
\end{Theorem}

To facilitate the proof of Theorem \ref{BKS2}, we first establish the following lemma. 
\begin{Lemma}
\label{lem:slow:ud}
Let $r(x) \in {\mathbf U}$ satisfy
\begin{equation}
\label{cond:lem:slow:ud}
 \lim_{x \rightarrow \infty} \frac{|r(x)|}{\log x} < \infty \quad \text{and} \quad \lim_{x \rightarrow \infty} \frac{|r(x)|}{\log W(x)} = \infty.
 \end{equation}
Then, for any $a, d \in \mathbb{N}$, the sequence $(r(an +d))_{n \in \mathbb{N}}$ is $w(n)$-uniformly distributed $\bmod \, 1$. In addition, for any positive integer $a$, the sequence $(r(a n \log n))_{n \in \mathbb{N}}$ is $w(n)$-uniformly distributed $\bmod \, 1$.
\end{Lemma}
\begin{proof}
Replacing $r(x)$ by $-r(x)$,  if necessary, we may assume that $r(x)$  is eventually non-negative.
First, we prove that $(r(n))_{n \in \mathbb{N}}$ is $w(n)$-uniformly distributed $\bmod \, 1$. 
It suffices to show that $r(x)$ satisfies conditions $(i)-(iv)$ of Theorem \ref{lem:Tsu2}. 
Conditions $(i)$ and $(iii)$ are straightforward. 
Condition $(ii)$ holds since 
\[ \lim_{x \rightarrow \infty} x r'(x) = \lim_{x \rightarrow \infty} \frac{r(x)}{\log x} < \infty.\]
Finally, condition $(iv)$ follows from 
\begin{equation} 
\lim_{x \rightarrow \infty} \frac{r'(x)}{w(x)} W(x) =  \lim_{x \rightarrow \infty} \frac{r(x)}{\log W(x)} = \infty.
\end{equation}
To complete the proof, we note that by  Theorem \ref{lem3}, both  $r(ax+d)$ and $r(ax \log x)$ satisfy \eqref{cond:lem:slow:ud}. 
Consequently, $((r(an +d))_{n \in \mathbb{N}}$ and $(r(a n \log n))_{n \in \mathbb{N}}$ are $w(n)$-uniformly distributed $\bmod \, 1$. 
\end{proof}

\begin{proof}[Proof of Theorem \ref{BKS2}]
Let us first show that $(1), (2), (3)$ and $(7)$ are equivalent. 

$(7) \Rightarrow (1)$: Consider first the case when for any $q(x) \in \mathbb{Q}[x]$,
\begin{equation}
 \lim_{x \rightarrow \infty} \frac{|u(x) - q(x)|}{\log x} = \infty.
\end{equation}
By Theorem \ref{BKS:prev}, $(u(n))_{n \in \mathbb{N}}$ is uniformly distributed $\bmod \, 1$ and so it follows from Theorem \ref{lem:Tsu1} that $(u(n))_{n \in \mathbb{N}}$ is also $w(n)$-uniformly distributed $\bmod \, 1$.
Now consider the case when for some $q(x) \in \mathbb{Q}[x]$,
\begin{equation}
\label{eq2:proof:main}
 \lim_{x \rightarrow \infty} \frac{|u(x) - q(x)|}{\log x} < \infty.
\end{equation}
Let $r(x) = u(x) - q(x)$. 
Then,
 \begin{equation*}
\lim_{x \rightarrow \infty} \frac{|r(x)|}{\log x} < \infty, \quad \lim_{x \rightarrow \infty} \frac{|r(x)|}{\log W(x)} = \infty.
 \end{equation*}
Since $q(x) \in \mathbb{Q}[x]$, for a suitably chosen $a \in \mathbb{N}$, there exists a partition $\mathbb{Z} = \bigcup\limits_{d=1}^{a} (a \mathbb{Z} + d)$ such that $q(n)$ is constant on each element of the partition.  
By Lemma \ref{lem:slow:ud}, $(r(an+ d))_{n \in \mathbb{N}}$  is $w(n)$-uniformly distributed $\bmod \, 1$, so $(u(n))_{n \in \mathbb{N}}$ is $w(n)$-uniformly distributed $\bmod \, 1$. 
This proves the implication $(7) \Rightarrow (1)$.

$(1) \Rightarrow (7)$: We  prove this by contradiction. Suppose that for some $q(x) \in \mathbb{Q}[x]$,
\begin{equation}
\label{eq:(5)}
 \lim_{x \rightarrow \infty} \frac{|u(x) - q(x)|}{\log W(x)} < \infty.
\end{equation}
Let $r(x) = u(x) - q(x)$.  
Find $k \in \mathbb{N}$ such $k q(n) \equiv 0 \bmod \bmod \,1$, so that $k r(n) \equiv k u(n) \bmod \bmod \,1$. Since $(u(n))_{n \in \mathbb{N}}$ is $w(n)$-uniformly distributed $\bmod \, 1$, so is $(ku(n))_{n \in \mathbb{N}}$. However, in view of \eqref{eq:(5)}, Theorem \ref{lem:Nie} implies that $(k r(n))_{n \in \mathbb{N}}$ is not $w(n)$-uniformly distributed $\bmod \, 1$, leading to a contradiction. 

$(1) \Leftrightarrow (2) \Leftrightarrow (3)$: 
Since $u(x) \in \mathbf{U}$ is a subpolynomial, it follows that $u(ax + d)$ is also subpolynomial and belongs to $\mathbf{U}$ (see the discussion preceding Theorem \ref{lem3}).
Since $(1) \Leftrightarrow (7)$, to verify $(1) \Rightarrow (2)$ , it suffices to show that for every $q(x) \in \mathbb{Q}[x]$,
\begin{equation}
\label{eq:4.68:sec2}
\lim_{x \to \infty} \frac{|u(ax + d) - q(x)|}{\log W(x)} = \infty.
\end{equation}
Note that
\begin{equation}
\label{eq:4.69:sec2}
 \lim_{x \rightarrow \infty} \frac{|u(ax+d) - q(x)|}{\log W(x)} = \lim\limits_{x \rightarrow \infty} \frac{|u(x) - q ((x-d)/a)|}{W ((x-d)/a)}
 \end{equation}
Since $\lim_{x \to \infty} \frac{W(x)}{W((x - d)/a)} = 1$ by Theorem \ref{lem3}, the limit in \eqref{eq:4.69:sec2} is infinite.
Thus, \eqref{eq:4.68:sec2} follows.

$(2) \Rightarrow (3)$ is immediate. 

$(3) \Rightarrow (1)$: Since $(1) \Leftrightarrow (7)$, the implication in question follows from the fact that for any $q(x) \in \mathbb{Q}[x]$,
\[ \lim_{x \rightarrow \infty} \frac{|u(x) -q (x)|}{\log W(x)} =  \lim_{x \rightarrow \infty} \frac{|u(ax +d) -q (ax+d)|}{\log W(ax+d)} =  \lim_{x \rightarrow \infty} \frac{|u(ax +d) -q (ax+d)|}{\log W(x)} = \infty, \]
where the second equality follows from $\lim_{x \to \infty} \frac{W(ax+d)}{W(x)} = 1$, by Theorem \ref{lem3}.

Now we prove that $(4), (5), (6)$ and $(7)$ are equivalent.
Let us start with showing that $(7)$ implies  $(4), (5), (6)$. 
Consider first the case when for any $q(x) \in \mathbb{Q}[x]$,
\begin{equation}
\label{eq:proof:main:prime}
 \lim_{x \rightarrow \infty} \frac{|u(x) - q(x)|}{\log x} = \infty.
\end{equation}
By Theorem \ref{BKS:prev},  $(u(p_n))_{n \in \mathbb{N}}$ is uniformly distributed $\bmod \, 1$ and so by Theorem \ref{lem:Tsu1}, it is also $w(n)$-uniformly distributed $\bmod \, 1$.
Additionally, since $ku(x) + \frac{(x-d)j}{a}$ satisfies \eqref{eq:proof:main:prime} for any positive integers $a, d$, any nonzero integer $k$ and for any $j = 1, 2 \dots, a$, the sequence $\left(k  u(p_n) + \frac{(p_n - d)j}{a} \right)_{n \in \mathbb{N}}$ is uniformly distributed $\bmod \, 1$.  

We will use the classical identity (see p.34 in \cite{Mo}), which states that for any integer $n$ and any positive integer $q \geq 2$, one has
\begin{equation}
\label{ex:orthogoanl}
\frac{1}{q} \sum_{j=1}^q e^{2 \pi i  \frac{(n-t)j}{q}} = 
\begin{cases}
	1, & \text{if $n \equiv t \, \bmod \, q$}\\
         0, & \text{otherwise}.
		 \end{cases}
\end{equation}
Thus we have that for any real sequence $(x_n)_{n \in \mathbb{N}}$, 
\begin{equation}
\label{eq:2.20:primetrick:new}
\sum_{\substack{ p \leq N \\ p \equiv t \, \bmod \, q} } e^{2 \pi i  x_p} 
= \sum_{ p \leq N  } e^{2 \pi i x_p}  \frac{1}{q} \sum_{j=1}^q e^{2 \pi i \frac{(p-t)j}{q}}
=\frac{1}{q} \sum_{j=1}^q \sum_{p \leq N} e^{2 \pi i \left(x_p + \frac{(p-t)j}{q} \right)}.
\end{equation}
Applying \eqref{eq:2.20:primetrick:new}, we have
\begin{align*}
\lim_{N \rightarrow \infty} \frac{1}{\pi_{a,d} (N)} \sum_{\substack{p \leq N \\ p \equiv d \bmod a}} e^{2 \pi i k u(p)}
&= \lim_{N \rightarrow \infty} \frac{1}{\pi_{a,d} (N)} \sum_{p \leq N } e(2 \pi i k u(p)) \frac{1}{a} \sum_{j=1}^a e^{2 \pi i \frac{(p - d)j}{a}} \\
&= \lim_{N \rightarrow \infty} \frac{\pi(N)}{\pi_{a,d} (N)} \frac{1}{a} \sum_{j=1}^a \frac{1}{\pi(N)} \sum_{p \leq N } e^{ 2 \pi i \left(k  u(p) + \frac{(p - d)j}{a} \right) } =0,
\end{align*}
since $\lim_{N \rightarrow \infty} \frac{\pi(N)}{\pi_{a,d}(N)} = \phi(a)$.   
Thus, $(u (p_n^{a,d}))_{n \in \mathbb{N}}$ is uniformly distributed $\bmod \, 1$ for any $a, d \in \mathbb{N}$ with $\gcd (a,d)=1$ and $1 \leq d \leq a$, and so by Theorem \ref{lem:Tsu1}, it is also $w(n)$-uniformly distributed $\bmod \, 1$.

Now consider the case when for some $q(x) \in \mathbb{Q}[x]$,
\begin{equation}
\label{eq2:proof:main}
 \lim_{x \rightarrow \infty} \frac{|u(x) - q(x)|}{\log x} < \infty.
\end{equation}
Let $r(x) = u(x) - q(x)$. 
Then,
 \begin{equation*}
 \lim_{x \rightarrow \infty} \frac{|r(x)|}{\log x} < \infty, \quad  \lim_{x \rightarrow \infty} \frac{|r(x)|}{\log W(x)} = \infty.
 \end{equation*}
By Corollary \ref{cor:prime:replace} and Lemma \ref{lem:slow:ud}, for any $a, d \in \mathbb{N}$ with  $(a,d) =1$, $(r(p_n^{a,d}))_{n \in \mathbb{N}}$ is $w(n)$-uniformly distributed $\bmod \, 1$ and so is $(u(p_n^{a,d}))_{n \in \mathbb{N}}$. This implies that $(7) \Rightarrow (5)$. 
Since $(5) \Rightarrow (4)$ and $(5) \Rightarrow (6)$, we conclude that (7) also implies (4) and (6).

It remains to show that $(4) \Rightarrow (7)$, $(5) \Rightarrow (7)$ and $(6) \Rightarrow (7)$. 
First we show, by way of contradiction, that $(6) \Rightarrow (7)$.
Indeed, if, for some  $q(x) \in \mathbb{Q}[x]$,
\[ \lim_{x \rightarrow \infty} \frac{|u(x) - q(x)|}{\log W(x)} < \infty, \]
then for any $a, d \in \mathbb{N}$ with $\gcd (a,d)=1$, $(u (p_n^{a,d}))_{n \in \mathbb{N}}$ is not $w(n)$-uniformly distributed $\bmod \, 1$. 
Let $r(x) = u(x) - q(x)$. Take a non-zero integer $k$ such that $kq(n) \equiv 0 \bmod \mathbb{Z}$ for all $n$. 
Then, by Theorem \ref{lem3},
\[ \lim_{n \rightarrow \infty} \frac{|k r (\phi(a) n \log n)|}{ \log W(n)} = \lim_{x \rightarrow \infty} \frac{|k r( \phi(a) x \log x)|}{|r(x)|} \frac{|u(x) - q(x)|}{\log W(x)} < \infty.\]
By Theorem \ref{lem:Nie}, this implies that $(k r(\phi(a) n \log n))_{n \in \mathbb{N}}$ is not $w(n)$-uniformly distributed $\bmod \, 1$. 
Applying Corollary \ref{cor:prime:replace}, we conclude that $(k r(p_n^{a,d}))_{n \in \mathbb{N}}$ not $w(n)$-uniformly distributed $\bmod \, 1$, which further implies that $( u(p_n^{a,d}))_{n \in \mathbb{N}}$  is also not $w(n)$-uniformly distributed $\bmod \, 1$. This proves the implication $(6) \Rightarrow (7)$.

Now note that when $a=1, d=1$, $p_n^{a,d} = p_n$.  This gives the implication $(4) \Rightarrow (7)$.
Finally, since $(5) \Rightarrow (6)$, the implication $(5) \Rightarrow (7)$ also follows.
\end{proof}

For a subpolynomial function $u(x) \in \mathbf{U}$, Theorem \ref{BKS2} shows that $(u(n))_{n \in \mathbb{N}}$ is $w(n)$-uniformly distributed if and only if $(u(an+d))_{n \in \mathbb{N}}$ is $w(n)$-uniformly distributed. 
The following corollary of Theorem \ref{BKS2} shows that $(u(n))_{n \in \mathbb{N}}$ is $w(n)$-uniformly distributed if and only if it is $w(an+d)$-uniformly distributed.

\begin{Corollary}
\label{thm:equiv:new}
Let $u(x) \in \mathbf{U}$ be a subpolynomial function. Then the following are equivalent.
\begin{enumerate}
\item $(u(n))_{n \in \mathbb{N}}$ is $w(n)$-uniformly distributed.
\item $(u(n))_{n \in \mathbb{N}}$ is $w(an+d)$-uniformly distributed for any $a \in \mathbb{N}$ and any $d \in \{0, 1, \dots, a-1\}$.
\item $(u(p_n))_{n \in \mathbb{N}}$ is $w(an+d)$-uniformly distributed for any $a \in \mathbb{N}$ and any $d \in \{0, 1, \dots, a-1\}$.
\end{enumerate}
\end{Corollary}

\begin{proof}
Let $V(x) = \frac{1}{a} W (ax+d)$, so that $v(x) = V'(x) = w (ax+d)$. 
Since $\lim_{x \rightarrow \infty} \frac{\log W(x)}{\log x} < \infty$, Theorem \ref{lem3} implies that 
\[ \lim_{x \rightarrow \infty} \frac{\log W(x)}{\log V(x)} = 1.\]
Thus, the equivalence of (1) and (2) follows from the equivalence of (1) and (7) in Theorem \ref{BKS2}. 
Also,  the equivalence of (2) and (3) follows from the equivalence of (1) and (4) in Theorem \ref{BKS2}. 
\end{proof}

In the following theorem, we consider the distribution of the sequence $(u_n)_{n \in \mathbb{N}}$, where $u_n = \sum_{k=1}^n f(k)$ and $f(x) \in \mathbf{U}$ is a subpolynomial function.
This result will be used in Subsection \ref{subsec:3.4}.
\begin{Theorem}
\label{thm:sum:ud:new}
Let $f(x) \in \mathbf{U}$ be a subpolynomial. Assume that $f(x)$ is defined on $x \geq 0$. 
Define 
\[F(x) = \int_0^x f(t) \, dt\]
and
\[ u_n = \sum_{k = 1}^n f (k).  \]
Then the following are equivalent.
\begin{enumerate}
\item $(u_n)_{n \in \mathbb{N}}$ is $w(n)$-uniformly distributed $\bmod \, 1$. 
\item $(F(n))_{n \in \mathbb{N}}$ is $w(n)$-uniformly distributed $\bmod \, 1$. 
\item For any polynomial $q(x) \in \mathbb{Q}[x]$,
\begin{equation}
\label{cond:sum:ud:new}
 \lim_{x \rightarrow \infty} \frac{W(x)}{w(x)} |f(x) - q(x)| = \infty.
 \end{equation}
\end{enumerate}
\end{Theorem}

To prove Theorem \ref{thm:sum:ud:new}, we need the following lemma.

\begin{Lemma}
\label{lem:ap-unif}
Let $u(x) \in \mathbf{U}$ be subpolynomial, and suppose that $(u(n))_{n \in \mathbb{N}}$ is $w(n)$-uniformly distributed $\bmod \, 1$. 
Then, for any $q(x) \in \mathbb{Q}[x]$, both $(u(n) + q(n))_{n \in \mathbb{N}}$ and $(u(n) + \sum_{k=1}^n q(k))_{n \in \mathbb{N}}$ are $w(n)$-uniformly distributed $\bmod \, 1$.
\end{Lemma}

\begin{proof}
For $q(x) \in \mathbb{Q}[x]$, there exists $a \in \mathbb{N}$ such that for each $d = 0, 1, \dots, a-1 $, $\{q(an+d)\}$ is constant for all $n \in \mathbb{N}$. 
Since $(u(an+d))_{n \in \mathbb{N}}$ is  $w(n)$-uniformly distributed $\bmod \,1$ by Theorem \ref{BKS2}, and adding a constant does not affect $w(n)$-uniform distribution. 
Hence, $(u(an + d) + q(an + d))_{n \in \mathbb{N}}$ is $w(n)$-uniformly distributed $\bmod \, 1$ for each $d$. 
And so $(u(n) + q(n))_{n \in \mathbb{N}}$  is $w(n)$-uniformly distributed $\bmod \,1$.

Finally, for $q(x) \in \mathbb{Q}[x]$,  $\sum_{k=1}^n q(k) = Q(n)$ for some $Q(x) \in \mathbb{Q}[x]$, by a classical formula due to Bernoulli (cf. Section 7.9 of \cite{HarWri})
\[ \sum_{k=1}^n k^m =  \sum_{r=0}^m \frac{1}{m+1-r} \binom{m}{r} \beta_r (n+1)^{m+1-r}, \]
where $\beta_r$ (for $r = 0, 1, 2, \dots)$ is a sequence of rational numbers given by the generating function 
\[ \frac{x}{e^x -1} = \beta_0 + \frac{\beta_1}{1!} x + \frac{\beta_2}{2!} x^2  + \frac{\beta_3}{3!} x^3 + \cdots .\]
Then the result follows from the previous case applied to $Q(x) \in \mathbb{Q}[x]$.
\end{proof}

\begin{proof}[Proof of Theorem \ref{thm:sum:ud:new}]
Let $q(x) \in \mathbb{Q}[x]$. Let $q'(x) \in \mathbb{Q}[x]$ denote the derivative of $q(x)$.  
By L'Hospital's rule,
\[ \lim_{x \rightarrow \infty} \frac{|F(x) - q(x)|}{\log W(x)} = \lim_{x \rightarrow \infty} \frac{W(x)}{w(x)} |f(x) - q'(x)|.\]
Thus, equivalence of (2) and (3) follows from Theorem \ref{BKS2}. 
 
Let us  prove $(1) \Rightarrow (3)$. Suppose that there exists $q(x) \in \mathbb{Q}[x]$ such that $ \lim_{x \rightarrow \infty} \frac{W(x)}{w(x)} |f(x) - q(x)| < \infty$. 
Let $r(x) = f(x) - q(x)$. Then there exists $A > 0$ such that  $\sum_{k=1}^n |r(k)| \leq A \log W(n)$ for all $n$. Thus $(\sum_{k=1}^n r(k))_{n \in \mathbb{N}}$ is not $w(n)$-uniformly distributed $\bmod \,1$ by Theorem \ref{lem:Nie}. By Lemma \ref{lem:ap-unif}, $(x_n)_{n \in \mathbb{N}}$ is not $w(n)$-uniformly distributed $\bmod \,1$.

The proof for $(3) \Rightarrow (1)$ proceeds by considering separately the following two cases:
\begin{enumerate}[(i)]
\item There exists a polynomial $q_0(x) \in \mathbb{Q}[x]$ such that $f(x) - q_0(x) \ll \log W(x)$.
\item For any $q(x) \in \mathbb{Q}[x]$, 
\begin{equation}
\label{eq:cond2:sum-ud-new}
 \lim_{x \rightarrow \infty} \frac{|f(x) - q(x)|}{\log W(x)} = \infty.
 \end{equation}
\end{enumerate}

Let us first consider case (i).  
Let $r(x) = f(x) - q_0(x)$. Without loss of generality, we may assume that $r(x)$ is eventually non-negative, replacing $r(x)$ with $-r(x)$ if necessary. 
Now write 
\[ u_n = x_n + y_n, \]
where $x_n = F(n) + \sum_{k=1}^n q_0(k)$ and $y_n := \sum\limits_{k=1}^n \left( r(k) - \int_{k-1}^{k} r(t) \right) \, dt $. 
By Lemma \ref{lem:ap-unif}, $(x_n)_{n \in \mathbb{N}}$ is $w(n)$-uniformly distributed $\bmod \, 1$.
Since $r(x) \ll \log W(x)$, there exist $C > 0$ and $x_0 > 0$ such that $r'(x) \leq C \frac{w(x)}{W(x)}$ for all $x \geq x_0$. 
Then, for all $n \geq x_0$,
$$|y_n - y_{n+1}| = \left| \int_n^{n+1} r(n+1) - r(t) \, dt \right| \leq C \frac{w(n)}{W(n)}.$$ 
Thus, by Theorem \ref{thm:slowper},  $(u_n)_{n \in \mathbb{N}}$ is $w(n)$-uniformly distributed $\bmod \, 1$.

We now consider the second case. 
By Lemma 3.11 in \cite{Bos}, if  $f(x) \in \mathbf{U}$ is a subpolynomial, then there exists $p(x) \in \mathbb{Q}[x]$ such that one of the following holds:
\begin{enumerate}[(i)]
\item $\lim\limits_{x \rightarrow \infty} \frac{f(x) - p(x)}{x^n}$ is an irrational number for some $n \in \mathbb{N}$.
\item $\lim\limits_{x \rightarrow \infty} \frac{f(x) - p(x)}{x^{n+1}} = \lim\limits_{x \rightarrow \infty} \frac{x^n}{f(x) - p(x)} = 0$ for some $n \in \mathbb{N}$.
\item $\lim\limits_{x \rightarrow \infty} (f(x) - p(x))$ is a finite number.
\item $\lim\limits_{x \rightarrow \infty} (f(x) - p(x)) = \lim\limits_{x \rightarrow \infty} \frac{x}{f(x) - p(x)} = \pm \infty$
\end{enumerate}
Since case (iii) does not occur under the condition \eqref{eq:cond2:sum-ud-new}, we only need to consider cases (i), (ii) and (iv).

Suppose (i) or (ii) holds. For any $h \in \mathbb{N}$, let $S_h f (x) := f(x +1) + f(x+2) + \cdots + f(x+h)$. 
The conditions (i) or (ii) guarantee that for any $q(x)\in \mathbb{Q}[x]$,
\begin{equation} \lim_{x \rightarrow \infty} \frac{|S_h f (x) - q(x)|}{\log x} = \infty.\end{equation}
Therefore, by Theorem \ref{BKS:prev},  $(S_h f(n))_{n \in \mathbb{N}}$ is uniformly distributed $\bmod \, 1$ for any $h \in \mathbb{N}$. Note that $u_{n+h} - u_n = S_h f(n)$. 
Applying the classical van der Corput theorem (which is a special case of Corollary \ref{Cor:vdCdiff} for $w(n)=1$), we have that the sequence $(u_n)_{n \in \mathbb{N}}$ is uniformly distributed $\bmod \,1$ and so it is $w(n)$-uniformly distributed $\bmod \, 1$.  

It remains to consider the case (iv). In this case $|f(x) - p(x)| \prec x$. In addition, by \eqref{eq:cond2:sum-ud-new}, $\log W(x) \prec |f(x) - p(x)|$
Let $$S_h f(x) := f(x+1) + f(x+2) + \cdots + f(x+h) = \sum_{j=1}^h (f(x+j) - p(x+j)) + S_h p(x),$$
where $S_h p(x) = \sum_{j=1}^h p(x+j) \in \mathbb{Q}[x]$. 
Since $\log x \prec |f(x) - p(x)| \prec x$,  for any \( q(x) \in \mathbb{Q}[x] \),
\[
\lim_{x \rightarrow \infty} \frac{|S_h f(x) - q(x)|}{\log x} = \infty.
\]
Theorem \ref{BKS:prev} implies that $(S_h f(n))_{n \in \mathbb{N}}$ is uniformly distributed $\bmod \, 1$ for any $h \in \mathbb{N}$.
By the van der Corput difference theorem, $(u_n)_{n \in \mathbb{N}}$ is uniformly distributed $\bmod \, 1$ and so it is $w(n)$-uniformly distributed $\bmod \, 1$. 
\end{proof}

We conclude this section by discussing the notion of {\em $(w(n), \mu)$-uniform distribution} for  a Borel probability measure $\mu$ on $[0,1]$, introduced in \cite{N}. 
A sequence $(x_n)_{n \in \mathbb{N}}$ is said to be {\em $(w(n), \mu)$-uniformly distributed $\bmod \, 1$} if, for any interval $I = [a,b) \subset [0,1)$, 
\begin{equation*}
\lim_{N \rightarrow \infty} \frac{1}{ W(N)} \sum_{n=1}^N w(n) 1_{[a,b)} (\{ x_n \}) = \mu([a,b)).
\end{equation*}
This implies that for any continuous function $f$, 
\[ \lim_{N \rightarrow \infty} \frac{1}{W(N)} \sum_{n=1}^N w(n) f(x_n ) = \int f \, d \mu. \]
(See Theorem 1.8.3.1 in \cite{SP}. See also Theorem  1.2 of Chapter 3 in \cite{KN}.)
The next result shows that for a subpolynomial Hardy field function $u(x)$, the probability measure $\mu$ is either the Lebesgue measure or a discrete measure supported on a finite set.

\begin{Theorem}
\label{Cor:mainthmcor}
Let $u(x) \in \mathbf{U}$ be a subpolynomial function.
Suppose that $(u(n))_{n \in \mathbb{N}}$ is $(w(n), \mu)$-uniformly distributed $\bmod \, 1$ for some Borel probability measure $\mu$. 
Then $\mu$ is either the Lebesgue measure or a discrete measure supported on a finite set.
\end{Theorem}

\begin{proof}
If $\lim\limits_{x \rightarrow \infty} \frac{|u(x) - q(x)|}{\log x} = \infty$ for any $q(x) \in \mathbb{Q}[x]$, then $(u(n))_{n \in \mathbb{N}}$ is uniformly distributed $\bmod \, 1$ by Theorem \ref{BKS:prev}. 
Then, by Theorem \ref{lem:Tsu1}, $(u(n))_{n \in \mathbb{N}}$ is $w(n)$-uniformly distributed $\bmod \, 1$, and hence in this case the measure $\mu$ is the Lebesgue measure.

Now suppose that 
\[ \lim_{x \rightarrow \infty} \frac{|u(x) - q(x)|}{\log x} < \infty\]
for some $q(x) \in \mathbb{Q}[x]$. 
Then $u(x)$ can be expressed as $u(x) = q(x) + r(x)$, where $q(x) \in \mathbb{Q}[x]$ and 
\begin{equation}
\label{eq:r:slow}
\lim\limits_{x \rightarrow \infty} \frac{|r(x)|}{\log x} < \infty.
\end{equation}

We will consider two cases: $ \lim\limits_{x \rightarrow \infty} r(x)  \in \mathbb{R}$ and $\lim\limits_{x \rightarrow \infty} r(x) = \pm \infty$.
In the first case, since $q(n) \bmod \, 1$ is eventually periodic, the measure $\mu$ is a discrete measure supported on a finite set. 

Now consider the case when $\lim\limits_{x \rightarrow \infty} r(x) = \pm \infty$.  
We will show that $\lim\limits_{x \rightarrow \infty} \frac{|r(x)|}{\log W(x)} = \infty$. 
In view of \eqref{eq:r:slow}, this will imply that for any $q(x) \in \mathbb{Q}[x]$, 
$\lim\limits_{x \rightarrow \infty} \frac{|r(x) - q(x)|}{\log W(x)} = \infty$.
Then, applying Theorem \ref{BKS}, we will get that the measure $\mu$ is the Lebesgue measure. 

It remains to prove that if $\lim\limits_{x \rightarrow \infty} \frac{|r(x)|}{\log W(x)} < \infty$, then for some non-zero integer $m$, $\frac{1}{W(N)}  \sum\limits_{n=1}^N w(n) e^{2 \pi i m u(n)}$ fails to converge, and so $(u(n))_{n \in \mathbb{N}}$ is not $(w(n), \mu)$-uniformly distributed $\bmod \, 1$ for any Borel probability measure $\mu$.

Suppose, by way of contradiction, that $\lim\limits_{x \rightarrow \infty} \frac{|r(x)|}{\log W(x)} < \infty$. 
Let us assume for convenience that $r(x)$ is eventually non-negative. (In the case when $r(x)$ is eventually negative, the argument is analogous and is omitted.)
Then there exists $a \in [0, \infty)$ such that 
\begin{equation}
\label{eq:Cor2.8:cond:last}
 \lim_{x \rightarrow \infty} W(x) \frac{r'(x)}{w(x)} = \lim_{x \rightarrow \infty} \frac{r(x)}{\log W(x)} = a.
 \end{equation}

Let $m \in \mathbb{Z}$. 
By formulas \eqref{eq2.27:conv:sumtoint} and \eqref{eq2.10:sec2} which were obtained in the course of the proof of Theorem \ref{lem:estimate:expsum:Eulersum}, we have
\[  \lim_{N \rightarrow \infty} \left| \frac{1}{W(N)} \sum_{n=1}^N w(n) e^{2 \pi i m r(n)} - \frac{1}{W(N)} \int_1^N w(x) e^{2 \pi i m r(x)} \, dx \right| =0 \]
and 
\begin{align}
\label{eq2.10}
\frac{1+ 2 \pi i m a}{W(N)} \int_1^N w(x) e^{2 \pi i m r(x)} \, dx &= 
 \frac{1}{W(N)}  \left[  W(x) e^{2 \pi i m r(x)} \right]_1^N 
 - \frac{2 \pi i m}{W(N)} \int_1^N \left(\frac{W(x) r'(x)}{w(x)} - a \right) w(x) e^{2 \pi i m r(x)} dx \notag \\
&=\Sigma_1 - \Sigma_2,
\end{align}
where 
\[  \Sigma_1 = \frac{1}{W(N)}  \left[  W(x) e^{2 \pi i m r(x)} \right]_1^N  \quad \text{and} \quad  \Sigma_2 = \frac{2 \pi i m}{W(N)} \int_1^N \left(\frac{W(x) r'(x)}{w(x)} - a \right) w(x) e^{2 \pi i m r(x)} dx. \]
Note that for $m \ne 0$, $\Sigma_2$ tends to zero as $N \rightarrow \infty$ due to \eqref{eq:Cor2.8:cond:last} and 
$$\Sigma_1 = \frac{1}{W(N)}  \left[  W(x) e^{2 \pi i m r(x)} \right]_1^N  = e^{2 \pi i m r(N)} - \frac{W(1)}{W(N)} e^{2 \pi i m r(1)}$$ 
does not  converge as $N \rightarrow \infty$.
Thus, $\frac{1}{W(N)}  \sum\limits_{n=1}^N w(n) e^{2 \pi i m r(n)}$ fails to converge as $N \rightarrow \infty$.

Since $u(x) = q(x) + r(x)$ for some $q(x) \in \mathbb{Q}[x]$, there exists a non-zero integer $m$ such that $m u(n) \equiv m r(n) \bmod \, 1$ for all $n \in \mathbb{N}$. 
Thus, $\frac{1}{W(N)}  \sum\limits_{n=1}^N w(n) e^{2 \pi i m u(n)}$ also fails to converge. 
\end{proof}

\section{Applications}
\label{sec:app}
The main goal of this section is to present applications of the results established in the previous sections.
In Subsection \ref{subsec:3.1}, we introduce the notion of weighted uniform distribution for sequences in $\mathbb{T}^k$, obtain a multidimensional generalization of Theorem \ref{BKS}, and derive several corollaries. These results are then used to derive ergodic and combinatorial applications in Subsections \ref{ergodicaverages}--\ref{subsec:3.3}.
Finally, in Subsection \ref{subsec:3.4}, we obtain several new results on Benford's law.
Throughout this section, we write $e(x) = e^{2 \pi i x}$ for simplicity.

\subsection{Uniform distribution in $\mathbb{T}^k$}\label{subsec:3.1} 
Observe that Definition \ref{def:weighted-u.d.} deals, up to an obvious change of language, with weighted uniform distribution in the one-dimensional torus $\mathbb{T}$.
The notion of weighted uniform distribution can be further extended to compact groups as follows.
Let $(w(n))_{n \in \mathbb{N}}$ be a non-increasing, positive sequence such that $\lim\limits_{N \rightarrow \infty} W(N) = \infty$, where $W(N) = \sum\limits_{n=1}^N w(n)$.
Let $G$ be a compact topological group and let $\mu$ be a Borel probability measure on $G$. (In this paper, we are mainly interested in the cases where $G = \mathbb{T}^k$ or $G = \mathbb{T}^k \times \mathbb{Z}_b^m$, where $\mathbb{T} = \mathbb{R}/\mathbb{Z}$ and $\mathbb{Z}_b = \mathbb{Z}/ b\mathbb{Z}$ for some $b \in \mathbb{N}$.)
A sequence $(x_n)_{n \in \mathbb{N}} \subset G$ is said to be {\em $(w(n), \mu)$-uniformly distributed in $G$} if, for any continuous function $F$ on $G$,
\begin{equation}
\lim_{N \rightarrow \infty} \frac{1}{ W(N)} \sum_{n=1}^N w(n) F ( x_n ) = \int F \, d \mu.
\end{equation}
In particular, if $\mu$ is the Haar probability measure on $G$, then we say that $(x_n)_{n \in \mathbb{N}} \subset G$ is {\em $w(n)$-uniformly distributed in $G$}.
We also say that a sequence $(x_{1,n}, \dots, x_{k,n})_{n \in \mathbb{N}}$ in $\mathbb{R}^k$ is $w(n)$-uniformly distributed $\bmod \,1$ if
$( \{ x_{1,n} \}, \dots, \{ x_{k,n} \})_{n \in \mathbb{N}}$ is $w(n)$-uniformly distributed in $\mathbb{T}^k$.

The proof of the following theorem can be obtained by complete analogy with that of the classical results (see Theorems 6.2 and 6.3 of Section 1.6 in \cite{KN}) and is therefore omitted.

\begin{Theorem}
\label{WeylCriterion}
Let $(w(n))_{n \in \mathbb{N}}$ be a non-increasing, positive sequence such that $\lim\limits_{N \rightarrow \infty} W(N) = \infty$, where $W(N) = \sum\limits_{n=1}^N w(n)$. 
Let  $(x_n)_{n \in \mathbb{N}}$ be a sequence in $\mathbb{T}^k$, where  $x_n = (x_{1,n}, \dots, x_{k,n})$. Then the following are equivalent:
\begin{enumerate}
\item $(x_n)_{n \in \mathbb{N}}$ is $w(n)$-uniformly distributed in $\mathbb{T}^k$.
\item For any nonzero vector $h = (h_1, \dots, h_k) \in \mathbb{Z}^k$, 
\[ \lim_{N \rightarrow \infty} \frac{1}{N} \sum_{n=1}^N e \left(\langle h, x_n \rangle \right) =0, \]
where $\langle h, x_n \rangle = \sum_{j=1}^k h_j x_{j, n}$.
\item For any nonzero vector $h=(h_1, \dots, h_k) \in \mathbb{Z}^k$, the sequence $\left( \langle h, x_n \rangle\right)_{n \in \mathbb{N}} = \left( \sum_{j=1}^k h_j x_{j, n}  \right)_{n \in \mathbb{N}} $ is $w(n)$-uniformly distributed $\bmod \, 1$.
\end{enumerate}
\end{Theorem}

For the remainder of this section, unless otherwise specified, we assume that $W(x)$ is a Hardy field function in $\mathbf{E}$ such that  $\lim\limits_{x \rightarrow \infty} W(x) = \infty$ and  $w(x) := W'(x)$ is a non-increasing positive function, as in Section \ref{sec2}. 
The next theorem is an immediate consequence of Theorem \ref{BKS} and Theorem \ref{WeylCriterion}. 
\begin{Theorem}
\label{ud:d-dim}
Let  $u_1(x), \dots, u_k(x)$ be subpolynomial functions belonging to some Hardy field $\mathbf{H}$.

Then the following are equivalent. 
\begin{enumerate}
\item $(u_1(n), \dots, u_k(n))_{n \in \mathbb{N}}$ is $w(n)$-uniformly distributed $\bmod \,1$.
\item  For any $a, d \in \mathbb{N}$, $(u_1(an+d), \dots, u_k(an+d))_{n \in \mathbb{N}}$ is $w(n)$-uniformly distributed $\bmod \,1$.
\item  For some $a, d \in \mathbb{N}$, $(u_1(an+d), \dots, u_k(an+d))_{n \in \mathbb{N}}$ is $w(n)$-uniformly distributed $\bmod \,1$.
\item $(u_1(p_n), \dots, u_k(p_n))_{n \in \mathbb{N}}$ is $w(n)$-uniformly distributed $\bmod \,1$.
\item For any $a, d \in \mathbb{N}$ with $\gcd (a,d)=1$, $(u_1(p_n^{a,d}), \dots, u_k(p_n^{a,d}) )_{n \in \mathbb{N}}$ is $w(n)$-uniformly distributed $\bmod \,1$, where $p_n^{a,d}$ is the $n$-th prime in the arithmetic progression $\{am + d: m = 0, 1, 2, \dots\}$. 
\item For some $a, d \in \mathbb{N}$ with $\gcd (a,d)=1$, $(u_1(p_n^{a,d}), \dots, u_k(p_n^{a,d}) )_{n \in \mathbb{N}}$ is $w(n)$-uniformly distributed $\bmod \,1$.
\item For any $u(x) \in \text{span}_{\mathbb{Z}}^* \{ u_1(x), \dots, u_k(x) \}$,
\[ \lim_{x \rightarrow \infty} \frac{|u(x) - q(x)|}{\log W(x)} = \infty \quad \text{ for any } q(x) \in \mathbb{Q}[x].\]
\end{enumerate}
\end{Theorem}

Let us write $Q(x) \in \mathbb{Q}[x] + \mathbb{R}$ if all coefficients of the polynomial $Q(x) \in \mathbb{R}[x]$, except for the constant term, are rational. In other words, $Q(x) \in \mathbb{Q}[x] + \mathbb{R}$ if
$$Q(x) = a_m x^m + \cdots + a_1 x + a_0, \text{ where } a_m, \dots, a_1 \in \mathbb{Q} \text{ and } a_0 \in \mathbb{R}.$$
If $Q(x) \in \mathbb{Q}[x] + \mathbb{R}$, then there exists $M \in \mathbb{N}$ such that, for every $n \in \mathbb{N}$, the fractional part $\{ Q(n) \}$ depends only on the residue class of $n \, \bmod \, M$.
Consequently, there exists a probability measure $\nu_{Q}^{(1)}$ supported on a finite subset of $\mathbb{T}$ such that $(\{Q(n)\})_{n \in \mathbb{N}}$ is $(1, \nu_{Q}^{(1)})$-uniformly distributed: for any $1$-periodic continuous function $f$ on $\mathbb{R}$,
\[\lim_{N \rightarrow \infty} \frac{1}{N} \sum_{n=1}^N f (Q(n)) = \int f \, d \nu_{Q}^{(1)}.\]
Similarly, by the prime number theorem for arithmetic progressions, there exists a probability measure $\nu_{Q}^{(2)}$ supported on a finite subset of $\mathbb{T}$ such that $(\{Q(p_n)\})_{n \in \mathbb{N}}$ is $(1, \nu_{Q}^{(2)})$-uniformly distributed: for any $1$-periodic continuous function $f$ on $\mathbb{R}$,
\[\lim_{N \rightarrow \infty} \frac{1}{N} \sum_{n=1}^N f (Q(p_n)) = \int f \, d \nu_{Q}^{(2)}.\]

\begin{Corollary}
\label{cor:ud-multi}
Let 
\[x_n = (\{Q(n)\}, \{u_1(n)\}, \{u_2(n)\}, \dots, \{u_k(n)\}), n \in \mathbb{N},\]
where $Q(x) \in \mathbb{R}[x]$ and $u_1(x), \dots, u_k(x)$ are subpolynomial functions belonging to some Hardy field $\mathbf{H}$. 
\begin{enumerate}
\item Suppose $Q(x) \in \mathbb{Q}[x] + \mathbb{R}$. Assume that for any $u(x) \in \text{span}_{\mathbb{Z}}^* \{ u_1(x), \dots, u_k(x) \}$,\footnote{See \eqref{eq:def:span1} for the definition of $\text{span}_{\mathbb{Z}}^*$.}
\begin{equation}
\label{eq:cond-proof-cor}
 \lim_{x \rightarrow \infty} \frac{|u(x) - q(x)|}{\log W(x)} = \infty \quad \text{ for any } q(x) \in \mathbb{Q}[x].
\end{equation}
Then, $(x_n)_{n \in \mathbb{N}}$ is $(w(n), \nu_{Q}^{(1)} \times \lambda^k)$-uniformly distributed in $\mathbb{T}^{k+1}$ and $(x_{p_n})_{n \in \mathbb{N}}$ is $(w(n), \nu_{Q}^{(2)}\times \lambda^k)$-uniformly distributed in $\mathbb{T}^{k+1}$. (Here, $\lambda$ denotes the Lebesgue measure on $\mathbb{T}$.)
\item Suppose $Q(x) \notin \mathbb{Q}[x] + \mathbb{R}$. Assume that for any $u(x) \in \text{span}_{\mathbb{Z}}^* \{ Q(x), u_1(x), \dots, u_k(x) \}$,
\[ \lim_{x \rightarrow \infty} \frac{|u(x) - q(x)|}{\log W(x)} = \infty \quad \text{ for any } q(x) \in \mathbb{Q}[x].\]
Then both $(x_n)_{n \in \mathbb{N}}$ and $(x_{p_n})_{n \in \mathbb{N}}$ are $w(n)$-uniformly distributed in $\mathbb{T}^{k+1}$.
\end{enumerate}
\end{Corollary}

\begin{proof}
The claim in part (2) follows directly from Theorem \ref{ud:d-dim}.
Let us prove (1) for the sequence $(x_n)_{n \in \mathbb{N}}$. 
(The proof of (1) for the sequence $(x_{p_n})_{n \in \mathbb{N}}$ is analogous and therefore is omitted.)
It is enough to show that for any $(m_0, m_1, \dots, m_k) \in \mathbb{Z}^{k+1}$, 
\begin{equation} 
\label{eq:4.2-proof}
\lim_{N \rightarrow \infty} \frac{1}{W(N)} \sum_{n=1}^N w(n) e (m_0 Q(n) + \sum_{i=1}^k m_i u_i(n)) = \int e(m_0 x) \,  d \nu_{Q}^{(1)}  \cdot \prod_{i=1}^k \int e (m_i x) \, 
d \lambda.   
\end{equation}

If for some $i \in \{1, 2, \dots, k\}$, $m_i \ne 0$, then condition \eqref{eq:cond-proof-cor} implies that the sequence $\left(m_0 Q(n) + \sum\limits_{i=1}^k m_i u_i(n)\right)_{n \in \mathbb{N}}$ is $w(n)$-uniformly distributed $\bmod \, 1$, hence \eqref{eq:4.2-proof} holds. 

If $m_1 = m_2 = \cdots = m_k = 0$, then by Theorem \ref{lem:Tsu1}, 
 \begin{align*} 
&\lim_{N \rightarrow \infty} \frac{1}{W(N)} \sum_{n=1}^N w(n) e (m_0 Q(n) + \sum_{i=1}^k m_i u_i(n)) \\
&\quad = \lim_{N \rightarrow \infty} \frac{1}{W(N)} \sum_{n=1}^N w(n) e (m_0 Q(n) ) 
= \lim_{N \rightarrow \infty} \frac{1}{N} \sum_{n=1}^N  e (m_0 Q(n) )  = \int e(m_0 x) \,  d \nu_{Q}^{(1)}.
\end{align*}
Thus, \eqref{eq:4.2-proof} holds in this case as well. 
\end{proof}

The following proposition will be used in Subsections \ref{ergodicaverages} and \ref{subsec:setofrec}.
\begin{Proposition}
\label{prop:app:ud-ave}
Let 
\[ R(x) = Q(x) + \sum_{i=1}^k \alpha_i \lfloor u_i(x) \rfloor,\]
where $Q(x) \in \mathbb{R}[x]$, $\alpha_1, \dots, \alpha_k \in \mathbb{R}$ and $u_1(x), \dots, u_k(x)$ are subpolynomial functions belonging to some Hardy field $\mathbf{H}$.
\begin{enumerate}
\item  Suppose $Q(x) \in \mathbb{Q}[x] + \mathbb{R}$. Assume that for any $u(x) \in \text{span}_{\mathbb{R}}^* \{ u_1(x), \dots, u_k(x) \}$,
\begin{equation}
\label{eq:cond1-prop}
 \lim_{x \rightarrow \infty} \frac{|u(x) - q(x)|}{\log W(x)} = \infty \quad \text{ for any } q(x) \in \mathbb{Q}[x].
\end{equation}
Then, 
\begin{equation}
\label{eq1:prop:app:ud-ave}
\lim_{N \rightarrow \infty} \frac{1}{W(N)} \sum_{n=1}^N w(n) e(R(n)) = 
\begin{cases}
\int e(x) \, d \nu_{Q}^{(1)}, & \text{ if } \alpha_1, \dots, \alpha_k \in \mathbb{Z}, \\
0, & \text{ otherwise}
\end{cases}
\end{equation}
and 
\begin{equation}
\label{eq2:prop:app:ud-ave}
\lim_{N \rightarrow \infty} \frac{1}{W(N)} \sum_{n=1}^N w(n) e(R(p_n)) = 
\begin{cases}
\int e(x) \, d \nu_{Q}^{(2)}, & \text{ if } \alpha_1, \dots, \alpha_k \in \mathbb{Z}, \\
0, & \text{ otherwise.}
\end{cases}
\end{equation}
\item Suppose $Q(x) \notin \mathbb{Q}[x] + \mathbb{R}$. Assume that for any $u(x) \in \text{span}_{\mathbb{R}}^* \{ u_1(x), \dots, u_k(x) \}$ and $a \in \mathbb{Z}$,
\begin{equation}
\label{eq:cond2-prop}
 \lim_{x \rightarrow \infty} \frac{|u(x) + a Q(x) - q(x)|}{\log W(x)} = \infty \quad \text{ for any } q(x) \in \mathbb{Q}[x].
\end{equation}
\end{enumerate}
Then, 
\begin{equation}
\label{eq3:prop:app:ud-ave}
\lim_{N \rightarrow \infty} \frac{1}{W(N)} \sum_{n=1}^N w(n) e(R(n)) = \lim_{N \rightarrow \infty} \frac{1}{W(N)} \sum_{n=1}^N w(n) e(R(p_n)) = 0
\end{equation}
\end{Proposition}

\begin{proof}
We first prove (1).
Let us establish \eqref{eq1:prop:app:ud-ave}. (The proof for \eqref{eq2:prop:app:ud-ave} is analogous and will be omitted.)

The statement is immediate when all $\alpha_1,\dots,\alpha_k \in \mathbb{Z}$, so it suffices to consider the case where at least one $\alpha_i \notin \mathbb{Z}$. 
Note that if $\alpha_i \in \mathbb{Z}$, then the term $\alpha_i \lfloor u_i(n) \rfloor$ is integer-valued. 
By discarding such trivial terms and relabeling, we may assume without loss of generality that $\alpha_i \notin \mathbb{Z}$ for every $1 \le i \le k$.

Rearranging if necessary, we may further assume that there exists an integer $l$ with $0 \leq l \leq k$ such that $\alpha_1, \dots, \alpha_l \notin \mathbb{Q}$ and $\alpha_{l+1}, \dots, \alpha_k \in \mathbb{Q}$. 
For $i = l+1, \dots, k$, we write $\alpha_i = \frac{c_i}{b}$, where $c_{i} \in \mathbb{Z}$ and $b \in \mathbb{N}$.
Define $(y_n)_{n \in \mathbb{N}}$ by
\begin{equation}\label{eq3.5:prop}
 y_n = \left( \{ Q(n) \}, \{ \alpha_1 u_1(n) \}, \{u_1(n) \}, \dots, \{ \alpha_l u_l(n) \}, \{u_l(n)\}, \left\{ \frac{u_{l+1}(n)}{b} \right\}, \dots, \left\{ \frac{u_k(n)}{b} \right\} \right),
 \end{equation} 
In view of Corollary \ref{cor:ud-multi}, the condition \eqref{eq:cond1-prop} implies that the sequence $(y_n)_{n \in \mathbb{N}}$ is $(w(n),  \nu_{Q}^{(1)} \times \lambda^{l +k})$-uniformly distributed $\bmod \,1$.  

Let $\mu_b$ denote the normalized counting measure on $\mathbb{Z}_b$. 
For any intervals $I_j \subset [0,1)$, $j = 0, 1, 2, \dots, 2l$, and for any integers $c_1, \dots c_{k-l}$ in $\{0, 1, 2, \dots, b-1\}$, 
define
$$A := \prod_{i=0}^{2l} I_i \subset \mathbb{T}^{2l} \quad \text{and} \quad  B := A \times \prod_{j=1}^{k-l} \left[ \frac{c_j}{b}, \frac{c_j +1}{b} \right) \subset \mathbb{T}^{l+k}.$$
Let $$ x_n = (  \{Q(n)\}, \{\alpha_1 u_1(n)\}, \{u_1(n)\}, \dots, \{\alpha_l u_l (n)\}, \{u_l(n)\}) \in \mathbb{T}^{2l +1}. $$ 
Note that
$$x_n \in A \text{ and } \lfloor u_{l+1}(n) \rfloor = c_1, \dots, \lfloor u_k(n) \rfloor =c_k  \quad \text{if and only if} \quad y_n \in B.$$ 
Since $(y_n)_{n \in \mathbb{N}}$ is $(w(n),  \nu_{Q}^{(1)} \times \lambda^{l +k})$-uniformly distributed $\bmod \,1$,
\begin{equation}
\label{eq:seq:group:prop} 
(\{Q(n)\}, \{\alpha_1 u_1(n)\}, \{u_1(n)\}, \dots, \{\alpha_l u_l(n)\}, \{u_l(n)\}, \lfloor u_{l+1}(n) \rfloor, \dots, \lfloor u_k(n) \rfloor )
\end{equation}
is $(w(n),  \nu_{Q}^{(1)} \times \lambda^{2l} \times \mu_b^{k-l})$-uniformly distributed in $\mathbb{T} \times \mathbb{T}^{2l} \times \mathbb{Z}_b^{k - l}$.

Now write 
$$e ( R(n)) =  f_0 (Q(n)) \cdot \prod_{i=1}^l f_i ( \alpha_i u_i(n),  u_i(n)) \cdot \prod_{i=l+1}^k f_i( \lfloor u_i(n) \rfloor),$$
where $f_0(x) = e(x)$, $f_i(x,y) = e(x - \{y\} \alpha_i)$ for $i \leq l$ and $f_i(x) = e (\alpha_i x)$ for $i \geq l+1$. 
Since the sequence in \eqref{eq:seq:group:prop}  is  $(w(n),  \nu_{Q}^{(1)} \times \lambda^{2l} \times \mu_b^{k-l})$-uniformly distributed in $\mathbb{T}^{2l} \times \mathbb{Z}_b^{k - l}$, we have that
\begin{equation}
\label{eq:eval:average:R:prop} 
\lim_{N \rightarrow \infty} \frac{1}{W(N)} \sum_{n=1}^N w(n) e ( R(n)) = \int f_0 d \nu_{Q}^{(1)} \prod_{i=1}^l \int_{\mathbb{T}^2} e (m x - m \alpha_i \{y\}) dx dy \prod_{i=l+1}^k \frac{1}{b} \sum_{t=0}^{b-1} e (m \alpha_i t).
\end{equation}

First, assume that some $\alpha_i \notin \mathbb{Q}$, so $\alpha_1 \notin \mathbb{Q}$. Then, we have 
$$\int_{\mathbb{T}^2} e (m x - m \alpha_1 \{y\}) dx dy = \int_0^1 \int_0^1 e(m x) dx \, e(- m\alpha_1 \{y\}) dy = 0.$$
Therefore, the limit in \eqref{eq:eval:average:R:prop} equals zero. 

Now assume that all $\alpha_i \in \mathbb{Q}$. 
Define  $f_i(x) = e (\alpha_i x)$ for $i = 1, 2, \dots, k$. 
Then we obtain
\begin{equation}
\label{cor3.4:lasteq:prop}
\lim_{N \rightarrow \infty} \frac{1}{W(N)} \sum_{n=1}^N w(n) e(R(n)) 
=  \lim_{N \rightarrow \infty} \frac{1}{W(N)} \sum_{n=1}^N w(n) \prod_{i=1}^k f_i(\lfloor u_i(n) \rfloor ) = \prod_{i=1}^k \frac{1}{b} \sum_{t=0}^{b-1} e (\alpha_i t).
\end{equation}
Since some   $\alpha_i \notin \mathbb{Z}$, we have $\frac{1}{b} \sum\limits_{t=0}^{b-1} e (\alpha_i t) = 0$. 
So, the right-hand side, and hence the left-hand side, of formula \eqref{cor3.4:lasteq:prop} is zero.

We now prove (2). As in the proof of (1), condition \eqref{eq:cond2-prop} implies that 
\begin{equation}
\label{eq2:seq:group:prop} 
(\{Q(n)\}, \{\alpha_1 u_1(n)\}, \{u_1(n)\}, \dots, \{\alpha_l u_l(n)\}, \{u_l(n)\}, \lfloor u_{l+1}(n) \rfloor, \dots, \lfloor u_k(n) \rfloor )
\end{equation}
is $(w(n),  \lambda^{2l+1} \times \mu_b^{k-l})$-uniformly distributed in $\mathbb{T} \times \mathbb{T}^{2l} \times \mathbb{Z}_b^{k - l}$. 
Applying the same reasoning as in the proof of (1), we obtain \eqref{eq3:prop:app:ud-ave}. 
\end{proof}

The following theorem deals with the distribution $\bmod \, 1$ of the sequences of the form 
\begin{equation}
(u_1(n), \dots, u_k(n), v_1(p_n), \dots v_l(p_n))_{n \in \mathbb{N}},
\end{equation}
and 
\begin{equation}
(u_1(p_n), \dots, u_k(p_n), v_1(n), \dots v_l(n))_{n \in \mathbb{N}},
\end{equation}
where $u_i(x)$, $1 \leq i \leq k$, are subpolynomial Hardy functions and $v_j(x)$, $1 \leq j \leq l$, are slowly growing Hardy field functions.

\begin{Theorem}
\label{thm:w-ud:mixedindex}
Let $u_1(x), \dots, u_k(x) \in \mathbf{H}$  for some Hardy field $\mathbf{H}$ and let $v_1(x), \dots, v_l(x) \in {\mathbf U}$.   
Assume that $u_1(x), \dots, u_k(x)$ are subpolynomial functions  $\lim\limits_{x \rightarrow \infty} \frac{|v_j(x)|}{\log x} < \infty$ for $j = 1, 2, \dots, l$. 
\begin{enumerate}
\item Suppose that $v_i(x \log x) \in {\mathbf H}$ for $1 \leq i \leq l$. Then the following are equivalent:
\begin{enumerate}
\item $(u_1(n), \dots, u_k(n), v_1(n \log n), \dots, v_l(n \log n) )_{n \in \mathbb{N}}$ is $w(n)$-uniformly distributed $\bmod \, 1$ in $\mathbb{R}^{k+l}$.
\item $(u_1(n), \dots, u_k(n), v_1(p_n), \dots, v_l(p_n))_{n \in \mathbb{N}}$ is $w(n)$-uniformly distributed $\bmod \, 1$ in $\mathbb{R}^{k+l}$.
\item For any $u(x) \in \text{span}_{\mathbb{Z}}^* \{ u_1(x), \dots, u_k(x), v_1(x \log x), \dots, v_l(x \log x) \}$,
\[ \lim_{x \rightarrow \infty} \frac{|u(x) -q(x)|}{\log W(x)} = \infty \quad \text{for any } q(x) \in \mathbb{Q}[x].\] 
\end{enumerate}
\item  Suppose that $v_i \circ g (x) \in {\mathbf H}$ for $1 \leq i \leq b$, where $g(x)$ is the compositional inverse of $f(x)= x \log x$, that is, $g(x \log x) = x$. 
Then the following are equivalent:
\begin{enumerate}
\item $(u_1(n), \dots, u_k(n), v_1 \circ g (n), \dots, v_l \circ g (n) )_{n \in \mathbb{N}}$ is $w(n)$-uniformly distributed $\bmod \, 1$ in $\mathbb{R}^{k+l}$.
\item $(u_1(p_n), \dots, u_k(p_n), v_1(n), \dots, v_l(n))_{n \in \mathbb{N}}$ is $w(n)$-uniformly distributed $\bmod \, 1$ in $\mathbb{R}^{k+l}$.
\item For any $u(x) \in \text{span}_{\mathbb{Z}}^* \{ u_1(x), \dots, u_k(x), v_1 \circ g (x), \dots, v_l \circ g (x) \}$,
\[ \lim_{x \rightarrow \infty} \frac{|u(x) -q(x)|}{\log W(x)} = \infty \quad \text{for any } q(x) \in \mathbb{Q}[x].\] 
\end{enumerate}
\end{enumerate}
\end{Theorem}

The proof of Theorem \ref{thm:w-ud:mixedindex} follows immediately from the next lemma together with Theorem \ref{WeylCriterion}.

\begin{Lemma}
\label{lem:sum} 
Let $u(x) \in {\mathbf H}$ for some Hardy field ${\mathbf H}$ and let $v(x) \in \mathbf{U}$.
Assume that $u(x)$ is a subpolynomial function and $\lim\limits_{x \rightarrow \infty} \frac{|v(x)|}{\log x} < \infty$. 
\begin{enumerate}
\item Suppose that $v(x \log x) \in {\mathbf H}$. Then the following are equivalent:
\begin{enumerate}
\item $(u(n) + v (n \log n)_{n \in \mathbb{N}}$ is $w(n)$-uniformly distributed $\bmod \, 1$. 
\item $(u(n) + v (p_n))_{n \in \mathbb{N}}$ is $w(n)$-uniformly distributed $\bmod \, 1$.
\item For any $q(x) \in \mathbb{Q}[x]$, 
\[\lim_{x \rightarrow \infty} \frac{|u(x) + v (x \log x) - q(x)|}{\log W(x)} = \infty.\]
\end{enumerate}
\item Suppose that $v \circ g (x) \in {\mathbf H}$, where $g(x)$ is the compositional inverse of $f(x) = x \log x$.. Then the following are equivalent: 
\begin{enumerate}
\item $(u(n) + v \circ g (n))_{n \in \mathbb{N}}$ is $w(n)$-uniformly distributed $\bmod \, 1$. 
\item $(u(p_n) + v (n))_{n \in \mathbb{N}}$ is $w(n)$-uniformly distributed $\bmod \, 1$.  
\item For any $q(x) \in \mathbb{Q}[x]$, 
\[\lim_{x \rightarrow \infty} \frac{|u(x) + v \circ g (x) - q(x)|}{\log W(x)} = \infty.\]
\end{enumerate}
\end{enumerate}
\end{Lemma}
\begin{proof}
Let us first prove (1). By Theorem \ref{BKS}, items (a) and (c) in (1) are equivalent and by Theorem \ref{lem2:sec2}, items (a) and (b) in (1) are equivalent. So, (1) follows.

Now we prove (2).
Note that by Theorem \ref{BKS}, $(u(n) + v \circ g (n))_{n \in \mathbb{N}}$  is $w(n)$-uniformly distributed $\bmod \, 1$ if and only if $(u(p_n) + v \circ g (p_n))_{n \in \mathbb{N}}$ is $w(n)$-uniformly distributed $\bmod \, 1$.  
Now, Theorem \ref{lem2:sec2} implies that $(u(p_n) + v \circ g (p_n))_{n \in \mathbb{N}}$ is $w(n)$-uniformly distributed $\bmod \, 1$
if and only if $(u(p_n) + v \circ g (n \log n))_{n \in \mathbb{N}}$ is $w(n)$-uniformly distributed $\bmod \, 1$. 
Since $v(n) = v \circ g (n \log n)$, items (a) and (b) in (2) are equivalent. 
Furthermore, equivalence of items (a) and (c) in (2) follows from Theorem \ref{BKS}.
\end{proof}

The following generalization of Proposition \ref{prop:app:ud-ave} will be used in Subsection \ref{ergodicaverages}. 
\begin{Proposition}
\label{prop:app:ud-ave2}
Let $u_1(x), \dots, u_k(x) \in \mathbf{H}$ for some Hardy field $\mathbf{H}$ and let $v_1(x), \dots, v_l(x) \in {\mathbf U}$. 
Assume that  $u_1(x), \dots, u_k(x)$ are subpolynomial functions  and  $\lim\limits_{x \rightarrow \infty} \frac{|v_j(x)|}{\log x} < \infty$ for $j = 1, 2, \dots, l$.

\begin{enumerate}
\item Suppose that $v_1(x \log x), \dots, v_l(x \log x) \in \mathbf{H}$.
Let 
\[ R_1(n) = Q(n) + \sum_{i=1}^k \alpha_i \lfloor u_i(n) \rfloor + \sum\limits_{j=1}^l \beta_j \lfloor v_j(p_n) \rfloor,\]
where $\alpha_1, \dots, \alpha_k, \beta_1, \dots, \beta_l \in \mathbb{R}$ and $Q(x) \in \mathbb{R}[x]$. 
\begin{enumerate}
\item  Suppose $Q(x) \in \mathbb{Q}[x] + \mathbb{R}$. Assume that for any $u(x) \in \text{span}_{\mathbb{R}}^* \{ u_1(x), \dots, u_k(x), v_1(x \log x), \dots, v_l(x \log x)\}$,
\begin{equation}
\label{eq:cond1-prop2}
 \lim_{x \rightarrow \infty} \frac{|u(x) - q(x)|}{\log W(x)} = \infty \quad \text{ for any } q(x) \in \mathbb{Q}[x].
\end{equation}
Then, 
\begin{equation}
\label{eq1:prop:app:ud-ave2}
\lim_{N \rightarrow \infty} \frac{1}{W(N)} \sum_{n=1}^N w(n) e(R_1(n)) = 
\begin{cases}
\int e(x) \, d \nu_{Q}^{(1)}, & \text{ if } \alpha_1, \dots, \alpha_k \in \mathbb{Z}, \\
0, & \text{ otherwise}.
\end{cases}
\end{equation}
\item Suppose $Q(x) \notin \mathbb{Q}[x] + \mathbb{R}$. Assume that for any $u(x) \in \text{span}_{\mathbb{R}}^* \{ u_1(x), \dots, u_k(x), v_1(x \log x), \dots, v_l(x \log x) \}$ and $a \in \mathbb{Z}$,
\begin{equation}
\label{eq:cond2-prop2}
 \lim_{x \rightarrow \infty} \frac{|u(x) + a Q(x) - q(x)|}{\log W(x)} = \infty \quad \text{ for any } q(x) \in \mathbb{Q}[x].
\end{equation}
\end{enumerate}
Then, 
\begin{equation}
\label{eq3:prop:app:ud-ave2}
\lim_{N \rightarrow \infty} \frac{1}{W(N)} \sum_{n=1}^N w(n) e(R_1(n)) = 0.
\end{equation}
\item Suppose that $ v_1 \circ g (x), \dots, v_l \circ g (x) \in \mathbf{H}$, where $g(x)$ is the compositional inverse of $f(x) = x \log x$.
Let 
\[ R_2(n) = Q(p_n) + \sum_{i=1}^k \alpha_i \lfloor u_i(p_n) \rfloor + \sum\limits_{j=1}^l \beta_j \lfloor v_j(n) \rfloor,\]
where $\alpha_1, \dots, \alpha_k, \beta_1, \dots, \beta_l \in \mathbb{R}$ and $Q(x) \in \mathbb{R}[x]$. 
\begin{enumerate}
\item  Suppose $Q(x) \in \mathbb{Q}[x] + \mathbb{R}$. Assume that for any $u(x) \in \text{span}_{\mathbb{R}}^* \{ u_1(x), \dots, u_k(x), v_1 \circ g(x), \dots, v_l \circ g (x)\}$,
\begin{equation}
\label{eq:cond1-prop2}
 \lim_{x \rightarrow \infty} \frac{|u(x) - q(x)|}{\log W(x)} = \infty \quad \text{ for any } q(x) \in \mathbb{Q}[x].
\end{equation}
Then, 
\begin{equation}
\label{eq1:prop:app:ud-ave3}
\lim_{N \rightarrow \infty} \frac{1}{W(N)} \sum_{n=1}^N w(n) e(R_1(n)) = 
\begin{cases}
\int e(x) \, d \nu_{Q}^{(2)}, & \text{ if } \alpha_1, \dots, \alpha_k \in \mathbb{Z}, \\
0, & \text{ otherwise}.
\end{cases}
\end{equation}
\item Suppose $Q(x) \notin \mathbb{Q}[x] + \mathbb{R}$. Assume that for any $u(x) \in \text{span}_{\mathbb{R}}^* \{ u_1(x), \dots, u_k(x), v_1 \circ g(x), \dots, v_l \circ g (x) \}$ and $a \in \mathbb{Z}$,
\begin{equation}
\label{eq:cond2-prop2}
 \lim_{x \rightarrow \infty} \frac{|u(x) + a Q(x) - q(x)|}{\log W(x)} = \infty \quad \text{ for any } q(x) \in \mathbb{Q}[x].
\end{equation}
\end{enumerate}
Then, 
\begin{equation}
\label{eq3:prop:app:ud-ave3}
\lim_{N \rightarrow \infty} \frac{1}{W(N)} \sum_{n=1}^N w(n) e(R_2(n))  = 0.
\end{equation}
\end{enumerate}
\end{Proposition}

To prove Proposition \ref{prop:app:ud-ave2}, we need the following lemma. It shows that if $u(x)$ is a Hardy field function satisfying suitable growth conditions, then the values $\lfloor u(p_n) \rfloor$ and $\lfloor u(n \log n) \rfloor$ differ only on a set of indices whose contribution to the weighted average is asymptotically negligible.

\begin{Lemma}
\label{lem:comp:slow:prime}
Suppose $u(x) \in {\mathbf U}$ satisfies
\begin{equation}
\label{eq:lem3.8}
(1) \lim\limits_{x \rightarrow \infty} \frac{|u(x)|}{\log x} < \infty \quad \text{ and } \quad (2) \lim\limits_{x \rightarrow \infty} \frac{|u(x)|}{ \log W(x)} = \infty.
\end{equation}
Define $D := \{ n: \lfloor u(n \log n) \rfloor \ne \lfloor u(p_n) \rfloor \}.$
Then, 
\[ \lim_{N \rightarrow \infty} \frac{1}{W(N)} \sum_{n=1}^N w(n) 1_{D} (n) = 0.\]
\end{Lemma}

\begin{proof}
Observe that it follows from formula \eqref{eqn:comp:prime-log} that condition (1) in \eqref{eq:lem3.8} implies that $\lim\limits_{n \rightarrow \infty} |u(p_n) - u (n \log n)| = 0$. 
Therefore, for any $\epsilon > 0$, there exists $N_0$ such that for all $n \geq N_0$, $| u(p_n) - u(n \log n) | < \epsilon$.
Let $$B(\epsilon) = \{ n \geq N_0: \{ u(p_n) \} \in (1 - \epsilon, 1) \cup [0, \epsilon) \}.$$
Note that   conditions (1) and (2) in \eqref{eq:lem3.8} imply
\[ \lim_{x \rightarrow \infty} \frac{|u(x) - q(x)|}{\log W(x)} = \infty \quad \text{ for any } q(x) \in \mathbb{Q}[x],\]
and so  $(u(p_n))_{n \in \mathbb{N}}$ is $w(n)$-uniformly distributed $\bmod \, 1$ by Theorem \ref{BKS}. 
Therefore,  
\[ \lim_{N \rightarrow \infty} \frac{1}{W(N)} \sum_{n=1}^N w(n) 1_{B(\epsilon)} (n) \leq  2 \epsilon.\]
Since $D \cap [N_0, \infty) \subset B(\epsilon)$, we obtain
\begin{equation} 
\label{eq:4.25}
\limsup_{N \rightarrow \infty} \frac{1}{W(N)} \sum_{n=1}^N w(n) 1_{D} (n) \leq 2 \epsilon. \end{equation}
As \eqref{eq:4.25} holds for any $\epsilon > 0$, the result follows. 
\end{proof}

\begin{proof}[Proof of Proposition \ref{prop:app:ud-ave2}]
We will prove (1). (The proof for (2) is analogous and is omitted.)
Define 
$$\tilde{R_1} (n) :=  Q(n) + \sum_{i=1}^k \alpha_i \lfloor u_i(n) \rfloor + \sum_{j=1}^l \beta_j \lfloor v_j(n \log n) \rfloor.$$
Let $D = \{n: \lfloor v_j(p_n) \rfloor \ne \lfloor v_j(n \log n) \rfloor \text{ for some } j \in \{1, 2, \dots, l\} \}$.
Since $\lim\limits_{x \rightarrow \infty} \frac{|v_j(x)|}{\log x} < \infty$ and $\lim\limits_{x \rightarrow \infty} \frac{|v_j(x)|}{\log W(x)} = \infty$, by Lemma \ref{lem:comp:slow:prime},  
 \[ \lim_{N \rightarrow \infty} \frac{1}{W(N)} \sum_{n=1}^N w(n) 1_{D} (n) = 0.\]
 So, we have that 
 \begin{equation}
\label{eq:comp:pfcorudmix:prop}
\lim_{N \rightarrow \infty} \left| \frac{1}{W(N)} \sum_{n=1}^N w(n) e(m {R_1}(n)) - \lim_{N \rightarrow \infty} \frac{1}{W(N)} \sum_{n=1}^N w(n) e(m \tilde{R_1}(n)) \right| =0.
\end{equation}
Then applying Proposition \ref{prop:app:ud-ave} to $\tilde{R}_1(n)$, we obtain \eqref{eq1:prop:app:ud-ave2} and \eqref{eq3:prop:app:ud-ave2}.
\end{proof}

\subsection{Ergodic averages for unitary $\mathbb{Z}^k$-actions}
\label{ergodicaverages}
In this subsection, we utilize the uniform distribution results from Subsection \ref{subsec:3.1} to obtain new ergodic theorems for unitary $\mathbb{Z}^k$-actions along sequences of the form $\lfloor u (n) \rfloor$ and $\lfloor u (p_n) \rfloor$.
An instrumental role in the proofs of the two theorems presented in this subsection is played by the following variant of the Bochner--Herglotz theorem. (See, for example, Section 1.4 in \cite{Ru}.)

\begin{Theorem}
\label{thm:BochnerHeglotz}
Let $U_1, \dots, U_k$ be commuting unitary operators on a Hilbert space $\mathcal{H}$. 
For any $f \in \mathcal{H}$, there is a measure $\nu_f$ on $\mathbb{T}^k$ such that 
$$  \langle U_1^{n_1} U_2^{n_2} \cdots U_k^{n_k} f , f \rangle  \, \, = \int_{\mathbb{T}^k} e (n_1 \gamma_1 + \cdots + n_k \gamma_k) \, d \nu_f (\gamma_1, \dots , \gamma_k), $$
for any $(n_1, n_2, \dots , n_k) \in \mathbb{Z}^k$. 
\end{Theorem}

Our first theorem in this subsection deals with weighted ergodic averages of expressions of the form 
$U_1^{\lfloor u_1(n) \rfloor} \cdots U_k^{ \lfloor u_k(n)\rfloor}$ and
$U_1^{\lfloor u_1(p_n) \rfloor} \cdots U_k^{ \lfloor u_k(p_n) \rfloor}$, where $U_1, \dots, U_k$ are commuting unitary operators. Note that formula \eqref{eq2.2:ergseq:prime} provides a generalization of Theorem 4.2 in \cite{BKS}, which corresponds to the case $w(x) = 1$.

\begin{Theorem}
\label{thm:ergthm}
Let $U_1, \dots, U_k$ be commuting unitary operators on a Hilbert space $\mathcal{H}$.
Let $P$ denote the orthogonal projection onto the invariant space $\mathcal{H}_{inv} = \{ f \in \mathcal{H} :  U_i f = f \,\, \textrm{for all} \,\, i = 1,2, \dots, k \}$. 
Let $u_1(x), \dots, u_k(x) $ be subpolynomial functions belonging to some Hardy field ${\mathbf H}$ and assume that for any $u(x) \in \text{span}_{\mathbb{R}}^* (u_1(x), \dots, u_k(x))$,  
\[ \lim_{x \rightarrow \infty} \frac{|u (x) - q(x)|}{\log W(x)} = \infty \quad \text{for any } q(x) \in \mathbb{Q}[x].\]
Then, for any $f \in \mathcal{H}$,
\begin{equation}
\label{eq2.2:ergseq}
\lim_{N \rightarrow \infty} \frac{1}{W(N)} \sum_{n=1}^N w(n) U_1^{\lfloor u_1(n) \rfloor} \cdots U_k^{ \lfloor u_k(n)\rfloor} f = Pf,
\end{equation}
and 
\begin{equation}
\label{eq2.2:ergseq:prime}
\lim_{N \rightarrow \infty} \frac{1}{W(N)} \sum_{n=1}^N w(n) U_1^{\lfloor u_1(p_n) \rfloor} \cdots U_k^{ \lfloor u_k(p_n) \rfloor} f = Pf.
\end{equation}
\end{Theorem}

\begin{proof}
Let us prove \eqref{eq2.2:ergseq}. 
We will use a Hilbert space splitting $\mathcal{H} = \mathcal{H}_{inv} \oplus \mathcal{H}_{erg}$, where 
\begin{align*}
\mathcal{H}_{inv} &= \{ f \in \mathcal{H} :  U_i f = f \,\, \textrm{for all} \,\, i = 1,2, \dots, k \}, \\
\mathcal{H}_{erg} &= \left\{ f \in \mathcal{H} : \lim_{N_1, \dots, N_k \rightarrow \infty} \left|\left|\frac{1}{N_1 \cdots N_k} \sum_{n_1=1}^{N_1} \cdots  \sum_{n_k=1}^{N_k}U_1^{n_1} \cdots U_k^{n_k} f \right|\right|_{\mathcal{H}} = 0  \right\}.
\end{align*}

Since ${U_1^{[u_1(n)]} \cdots U_k^{[u_k(n)]} f  = f}$ for $f \in \mathcal{H}_{inv}$, it suffices to prove that for $f \in \mathcal{H}_{erg}$, the left-hand side in \eqref{eq2.2:ergseq} converges to $0$. 
By the Bochner-Heglotz theorem, there is a measure $\nu_f$ on $\mathbb{T}^k$ such that for any $(n_1, n_2, \dots , n_k) \in \mathbb{Z}^k$,
$$  \langle U_1^{n_1} U_2^{n_2} \cdots U_k^{n_k} f , f \rangle  \, \, 
= \int_{\mathbb{T}^k} e (n_1 \gamma_1 + \cdots + n_k \gamma_k) \, d \nu_f (\gamma_1, \dots , \gamma_k) 
= \int_{\mathbb{T}^k} e ((n_1, \cdots, n_k) \cdot \gamma) \, d \nu_f (\gamma), $$
where $\gamma = (\gamma_1, \dots, \gamma_k)$.  
Note that $\nu_f (\{0,0 \dots, 0\})=0$, since $f \in \mathcal{H}_{erg}$.
Then, by Proposition \ref{prop:app:ud-ave} (see formula \eqref{eq1:prop:app:ud-ave}), we have 
\begin{align*}
&\lim_{N \rightarrow \infty}  \left|\left| \frac{1}{W(N)} \sum_{n=1}^N w(n) U_1^{\lfloor u_1(n) \rfloor } \cdots U_k^{ \lfloor u_k(n) \rfloor}  f \right|\right|_{\mathcal{H}}^2 \\
&=\lim_{N \rightarrow \infty}  \frac{1}{W(N)^2} \sum_{m,n=1}^N w(m) w(n) \langle U_1^{\lfloor u_1(m) \rfloor  - \lfloor u_1(n) \rfloor}  \cdots U_k^{\lfloor u_k(m) \rfloor -  \lfloor u_k(n) \rfloor } f, f \rangle \\
&=\lim_{N \rightarrow \infty}  \frac{1}{W(N)^2} \sum_{m,n=1}^N w(m) w(n) \int_{\mathbb{T}^k} e((\lfloor u_1(m) \rfloor -  \lfloor u_1(n) \rfloor , \dots ,  \lfloor u_k(m) \rfloor -  \lfloor u_k(n)\rfloor ) \cdot \bold{\gamma}) \, d \nu_f (\bold{\gamma}) \\
&= \lim_{N \rightarrow \infty}  \int_{\mathbb{T}^k} \left| \frac{1}{W(N)} \sum_{n=1}^N w(n) e ( (\lfloor u_1(n) \rfloor,  \dots , \lfloor u_k(n) \rfloor) \cdot \bold{\gamma}) \right|^2 \, d \nu_f(\bold{\gamma}) = \nu_f (\{0,0 \dots, 0\}) = 0.
\end{align*}
Thus, the proof of \eqref{eq2.2:ergseq} is complete. The proof of \eqref{eq2.2:ergseq:prime} is totally analogous, relying on formula  \eqref{eq2:prop:app:ud-ave} from Proposition \ref{prop:app:ud-ave} in place of formula \eqref{eq1:prop:app:ud-ave}. 
\end{proof}

The following theorem deals with mixed expressions of the form $U_1^{\lfloor u_1(n)\rfloor} \cdots U_k^{\lfloor u_k(n)\rfloor}U_{k+1}^{\lfloor v_1(p_n)\rfloor} \cdots U_{k+l}^{\lfloor v_l(p_n) \rfloor}$, as well as $U_1^{\lfloor u_1(p_n)\rfloor} \cdots U_k^{\lfloor u_k(p_n)\rfloor} U_{k+1}^{\lfloor v_1(n)\rfloor} \cdots U_{k+l}^{\lfloor v_l(n) \rfloor}$, and is a natural generalization of Theorem \ref{thm:ergthm}.
The proof is based on formulas \eqref{eq1:prop:app:ud-ave3} and \eqref{eq3:prop:app:ud-ave3} in Proposition \ref{prop:app:ud-ave2} and is omitted since it is similar to that of Theorem \ref{thm:ergthm}.

\begin{Theorem}
\label{thm:ergthm:mixed}
Let $U_1, \dots, U_{k+l}$ be commuting unitary operators on a Hilbert space $\mathcal{H}$. 
Let $P$ denote the orthogonal projection onto the invariant space $\mathcal{H}_{inv} = \{ f \in \mathcal{H} :  U_i f = f \,\, \textrm{for all} \,\, i = 1,2, \dots, k+l \}$.  
Assume that $u_1(x), \dots, u_k(x), v_1(x), \dots, v_l(x) \in \mathbf{U}$ are such that $u_1(x), \dots, u_k(x)$ are subpolynomial functions and $\lim\limits_{x \rightarrow \infty} \frac{|v_j(x)|}{\log x} < \infty$ for $j = 1, 2, \dots, l$. 
\begin{enumerate}
\item Suppose that  $u_1(x), \dots, u_k(x), v_1(x \log x), \dots, v_l(x \log x)$  belong to some Hardy field $\mathbf H$ and assume that for any $u(x) \in \text{span}_{\mathbb{R}}^* \{ u_1(x), \dots, u_k(x), v_1(x \log x), \dots, v_l(x \log x) \}$,
\[ \lim_{x \rightarrow \infty} \frac{|u(x) -q(x)|}{\log W(x)} = \infty \quad \text{ for any } q(x) \in \mathbb{Q}[x].\] 
Then for any $f \in \mathcal{H}$,
\begin{equation}
\lim_{N \rightarrow \infty} \frac{1}{W(N)} \sum_{n=1}^N w(n) U_1^{\lfloor u_1(n)\rfloor } \cdots U_k^{\lfloor u_k(n)\rfloor} \cdot U_{k+1}^{\lfloor v_1(p_n)\rfloor} \cdots U_{k+l}^{\lfloor v_{l} (p_n) \rfloor}  f = Pf.
\end{equation}
\item Suppose that $u_1(x), \dots, u_k(x), v_1 \circ g (x), \dots, v_l \circ g (x) $ belong to some Hardy field $\mathbf H$
and assume that for any $u(x) \in \text{span}_{\mathbb{R}}^* \{ u_1(x), \dots, u_k(x), v_1 \circ g (x), \dots, v_l \circ g (x) \}$, 
\[ \lim_{x \rightarrow \infty} \frac{|u(x) -q(x)|}{\log W(x)} = \infty \quad \text{ for any } q(x) \in \mathbb{Q}[x].\] 
Then for any $f \in \mathcal{H}$,
\begin{equation}
\lim_{N \rightarrow \infty} \frac{1}{W(N)} \sum_{n=1}^N w(n) U_1^{\lfloor u_1(p_n)\rfloor } \cdots U_k^{\lfloor u_k(p_n)\rfloor} \cdot U_{k+1}^{\lfloor v_1(n)\rfloor} \cdots U_{k+l}^{\lfloor v_{l} (n) \rfloor}  f = Pf.
\end{equation}
\end{enumerate}
\end{Theorem}

\subsection{Sets of recurrence in $\mathbb{Z}^k$}
\label{subsec:setofrec}
In this subsection, we obtain new results on sets of recurrence for measure preserving $\mathbb{Z}^k$-actions and derive several combinatorial corollaries.
First, we introduce the following definition.

\begin{Definition}
A set $D \subset \mathbb{Z}^k$ is a {\em set of recurrence} if for any commuting invertible measure preserving transformations $T_1, \dots, T_k$ on a probability space $(X, \mathcal{B}, \mu)$ and for any set $A \in \mathcal{B}$ with $\mu(A) > 0$, there exists $(d_1, \dots, d_k) \in D$, where not all of $d_i$ equal $0$, such that 
\[ \mu (A \cap T_1^{-d_1} \cdots T_k^{-d_k} A) > 0.\]
\end{Definition}

\begin{Theorem}
\label{thm:setofrec:H}
Let $q_1(x), \dots, q_m(x) \in \mathbb{Z}[x]$ be polynomials with $q_i(0) = 0$ for all $i = 1, 2, \dots, m$ and let $u_1(x), \dots, u_k(x) $ be subpolynomial functions  belonging to some Hardy field $\mathbf{H}$. 
Suppose that there exists a function $W(x) \in \mathbf{E}$ such that 
\begin{enumerate}[(i)]
\item $\lim\limits_{x \rightarrow \infty} W(x) = \infty$ and $w(x) := W'(x)$ is non-increasing and positive; 
\item for any $u(x) \in \text{span}_{\mathbb{R}}^* (u_1(x), \dots, u_k(x))$ and any $q(x) \in \mathbb{R}[x]$, 
\begin{equation}
\label{cond:eq:thm:setofrec:H}
\lim_{x \rightarrow \infty} \frac{|u(x) - q(x)|}{\log W(x)} = \infty.
\end{equation}
\end{enumerate}
Then the following sets are sets of recurrence in $\mathbb{Z}^{m+k}$.
\begin{enumerate}
\item $D_1= \{  (q_1(n), \dots, q_m(n), \lfloor u_1(n) \rfloor, \dots, \lfloor u_k(n) \rfloor ) : n \in \mathbb{N} \}$,
\item $D_2 = \{  (q_1(p-1), \dots, q_m(p-1), \lfloor u_1(p) \rfloor, \dots, \lfloor u_k(p) \rfloor ) : p \in \mathcal{P} \}$,
\item $D_3 = \{  (q_1(p+1), \dots, q_m(p+1), \lfloor u_1(p) \rfloor, \dots, \lfloor u_k(p) \rfloor ) : p \in \mathcal{P} \}$.
\end{enumerate}
\end{Theorem}

\begin{proof}
Let us first prove that $D_1$ is a set of recurrence. It suffices to show that for some positive integer $h$, 
\begin{equation}
\label{thm:setofrec:int}
\lim_{N \rightarrow \infty} \frac{h!}{W(N)} \sum_{n=1}^N w(n) 1_{h! \mathbb{Z}} (n) \mu (A \cap T_1^{-q_1( n)} \cdots T_m^{-q_m( n)} T_{m+1}^{- [u_1( n)]} \cdots T_{m+k}^{-[u_k(n)]} A) > 0.
\end{equation}

By Bochner-Herglotz theorem, there is a measure $\nu= \nu_f$, where $f = 1_A$, on $\mathbb{T}^{m+d}$ such that 
\begin{equation}
\label{eq:def:BHmeasure}
 \mu (A \cap T_1^{-n_1} \cdots T_{m+k}^{- n_{m+k}}  A) = \int_{\mathbb{T}^{m+k}} e (n_1 \gamma_1 + \cdots + n_{m+k} \gamma_{m+k}) \, d \nu 
 \end{equation}
for any $(n_1, \dots, n_{m+k}) \in \mathbb{Z}^{m+k}$. 
We will use below a classical result stating that  $\nu (\{(0,0, \dots, 0)\}) \geq \mu(A)^2$.
The following short proof of this fact is provided for the reader's convenience.
Applying the mean ergodic theorem for $\mathbb{Z}^{m+k}$-actions, we have
\begin{align*} 
\nu (\{(0,0, \dots, 0)\}) 
&=  \int_{\mathbb{T}^{m+k}}  \lim_{N \rightarrow \infty} \frac{1}{N^{m+k}} \sum_{n_1 = 1}^N \cdots \sum_{n_{m+k} =1}^N  e (n_1 \gamma_1 + \cdots + n_{m+k} \gamma_{m+k}) \, d \nu  \\
&=   \lim_{N \rightarrow \infty}   \frac{1}{N^{m+k}} \sum_{n_1 = 1}^N \cdots \sum_{n_{m+k} =1}^N  \mu (A \cap T_1^{-n_1} \cdots T_{m+k}^{- n_{m+k}}  A) \\
&=  \int_X 1_A \cdot \lim_{N \rightarrow \infty}   \frac{1}{N^{m+k}} \sum_{n_1 = 1}^N \cdots \sum_{n_{m+k} =1}^N   T_1^{n_1} \cdots T_{m+k}^{n_{m+k}}  1_A \,  d \mu = \int_X 1_A \cdot P 1_A \, d \mu,
\end{align*}
where $P$ is the orthogonal projection onto the invariant space $\mathcal{H}_{inv} = \{f \in L^2: T_i f = f \text{ for } i = 1, 2, \dots, m+k\}$.
Since $P$ is an orthogonal projection, we have
\begin{align*} 
\nu (\{(0,0, \dots, 0)\}) &= \int_X 1_A \cdot P 1_A \, d \mu = \int_X 1_A \cdot P^2 1_A \, d \mu = \int_X P^* 1_A \cdot P1_A \,d \mu \\ 
& =  \int_X P 1_A \cdot P1_A \,d \mu = \int_X (P1_A)^2  \, d \mu \int_X 1^2 \, d \mu \\ 
&\geq \left(\int_X P1_A \cdot 1 \, d \mu \right)^2 = \left(\int_X 1_A \cdot P^*1 \, d \mu \right)^2 = \left(\int_X 1_A  \, d \mu \right)^2 = \mu(A)^2.
\end{align*}
We will use a classical identity, that was already utilized in Section \ref{sec2}, which states  stating that for any integer $n$ and any positive integer $q \geq 2$,
\begin{equation}
\label{eq:cla-id:euler}
\frac{1}{q} \sum_{j=1}^q e \left( \frac{(n-t)j}{q}  \right) = 
\begin{cases}
	1, & \text{if $n \equiv t \, \bmod \, q$}\\
         0, & \text{otherwise}.
		 \end{cases}
\end{equation}
Using this identity with $q = h!$ and $t=0$, we rewrite the left-hand side of \eqref{thm:setofrec:int}  as 
\[ \lim_{N \rightarrow \infty}  \int F_N (\gamma_1, \dots, \gamma_{m+k}) \, d \nu, \]
where 
\begin{align*} F_N (\gamma_1, \dots, \gamma_{m+k}) &:= \frac{h!}{W(N)} \sum_{n=1}^N w(n) \, 1_{h!\mathbb{Z}} (n) e \left(\sum_{i=1}^m  q_i (n) \gamma_i + \sum_{i=1}^{k} [u_i( n)] \gamma_{m+i} \right) \\
&=  \sum_{j=1}^{h!} \frac{1}{W(N)} \sum_{n=1}^N w(n) \,  e \left(\sum_{i=1}^m  q_i (n) \gamma_i + \sum_{i=1}^{k} [u_i( n)] \gamma_{m+i}  + \frac{nj}{h!} \right).
\end{align*}
Note that Proposition \ref{prop:app:ud-ave} implies that $\lim\limits_{N \rightarrow \infty} F_N (\gamma_1, \dots, \gamma_{m+k})$ exists for every $(\gamma_1, \dots, \gamma_{m+k}) \in \mathbb{T}^{m+k}$. Hence the limit in \eqref{thm:setofrec:int} exists, and it remains to prove the inequality.

For each positive integer $h$, we consider the following partition of $\mathbb{T}^{m+k} $
\begin{equation}
\label{decomposition:torus}
 \mathbb{T}^{m+k} = I \cup J_h \cup K_h,
 \end{equation}
where 
\begin{enumerate}[(i)]
\item $(\gamma_1, \dots, \gamma_{m+k}) \in I \Leftrightarrow$  $\sum\limits_{j=1}^m \gamma_j q_j(n)$ has at least one irrational coefficient or $\gamma_j \notin \mathbb{Z}$ for some $j = m+1, m+2, \dots, m+k$,
\item $(\gamma_1, \dots, \gamma_{m+k}) \in J_h \Leftrightarrow$ $\sum\limits_{j=1}^m h! \gamma_j q_j(n)  \in \mathbb{Z}[n]$ and $ \gamma_j \in \mathbb{Z}$ for all $j = m+1, m+2, \dots, m+k$,
\item $K_h = \mathbb{T}^{m+k} \setminus (I \cup J_h)$.
\end{enumerate}
Since $K_h \supset K_{h+1}$ and $\bigcap\limits_{h=1}^{\infty} K_h = \emptyset$, one can choose $h$ large so that $\nu (K_h) < \frac{1}{2} \mu(A)^2$.

Note that if $(\gamma_1, \dots, \gamma_{m+k}) \in I$, then $\lim\limits_{N \rightarrow \infty} F_N (\gamma_1, \dots, \gamma_{m+k}) = 0$ by Proposition \ref{prop:app:ud-ave}, so 
\begin{equation}
\label{estimate:C}
\lim_{N \rightarrow \infty} \int_I F_N (\gamma_1, \dots, \gamma_{m+k}) d \nu = 0.
\end{equation}
Now consider $(\gamma_1, \dots, \gamma_{m+k}) \in J_h$. 
Then $F_N (\gamma_1, \dots, \gamma_{m+k}) =  \frac{h!}{W(N)} \sum\limits_{n=1}^N w(n) \, 1_{h!\mathbb{Z}} (n)$, so 
$$ \lim_{N \rightarrow \infty} F_N (\gamma_1, \dots, \gamma_{m+k}) =  \lim_{N \rightarrow \infty} \frac{h!}{W(N)} \sum_{n=1}^N w(n) \, 1_{h!\mathbb{Z}} (n) = \lim_{N \rightarrow \infty} \frac{h!}{N} \sum_{n=1}^N 1_{h!\mathbb{Z}} (n) =1.$$
 (Here the second equality is a consequence of Theorem \ref{lem:Tsu1}.)
Thus, we have that 
\begin{equation}
\label{estimate:A}
\lim_{N \rightarrow \infty} \int_{J_h} F_N (\gamma_1, \dots, \gamma_{m+k}) d \nu \geq \nu (\{ (0, 0, \dots, 0) \}) \geq \mu(A)^2.
\end{equation}
If $(\gamma_1, \dots, \gamma_{m+k}) \in K_h$, then $|F_N (\gamma_1, \dots, \gamma_{m+k})| \leq  \frac{h!}{W(N)} \sum\limits_{n=1}^N w(n) \, 1_{h!\mathbb{Z}} (n)$, so $\limsup\limits_{N \rightarrow \infty} |F_N (\gamma_1, \dots, \gamma_{m+k})| \leq 1$. Hence, 
\begin{equation}
\label{estimate:B}
\limsup_{N \rightarrow \infty} \int_{K_h} |F_N (\gamma_1, \dots, \gamma_{m+k})| d \nu \leq \nu (K_h) \leq \frac{1}{2} \mu(A)^2.
\end{equation}

Putting everything together from \eqref{estimate:C}, \eqref{estimate:A}, and \eqref{estimate:B}, we obtain
\begin{align*}
&\lim_{N \rightarrow \infty} \frac{h!}{W(N)} \sum_{n=1}^N w(n)  1_{h! \mathbb{Z}} (n)  \mu (A \cap T_1^{-q_1( n)} \cdots T_m^{-q_m( n)} T_{m+1}^{- [u_1( n)]} \cdots T_{m+k}^{-[u_k(  n)]} A) \\ 
& \quad \geq \lim_{N \rightarrow \infty} \left| \int_I F_N (\gamma_1, \dots, \gamma_{m+k}) d \nu \right| + \lim_{N \rightarrow \infty} \left| \int_{J_h} F_N (\gamma_1, \dots, \gamma_{m+k}) d \nu \right| - \limsup_{N \rightarrow \infty}\int_{K_h} \left |F_N (\gamma_1, \dots, \gamma_{m+k}) \right| \, d \nu \\
&\quad \geq \mu(A)^2 - \frac{1}{2} \mu(A)^2  >0,
\end{align*}
Thus we have shown that $D_1$ is a set of recurrence.

To prove that $D_2$ is a set of recurrence, we will show that for some large positive integer $h$,
\begin{equation}
\label{thm:setofrec:prime}
\lim_{N \rightarrow \infty} \frac{\phi(h!)}{W(N)} \sum_{n=1}^N w(n) 1_{h! \mathbb{Z} +1} (p_n) \mu (A \cap T_1^{-q_1(p_n-1)} \cdots T_m^{-q_m(p_n-1)} T_{m+1}^{- [u_1(p_n)]} \cdots T_{m+d}^{-[u_d(p_n)]} A)  > 0. 
\end{equation}

Let $\nu = \nu_f$, where $f = 1_A$, be the measure associated to $f$ via the Bochner-Herglotz representation as in \eqref{eq:def:BHmeasure}.

Using identity \eqref{eq:cla-id:euler} with $q = h!$ and $t=1$, we rewrite the left-hand side of \eqref{thm:setofrec:prime} as 
\[ \lim_{N \rightarrow \infty}  \int G_N (\gamma_1, \dots, \gamma_{m+d}) \, d \nu, \]
where 
\begin{align*} 
G_N (\gamma_1, \dots, \gamma_{m+d}) &:=  \frac{\phi (h!)}{W(N)} \sum_{n=1}^N w(n) \, 1_{h!\mathbb{Z} +1} (n) e \left(\sum_{i=1}^m  q_i (p_n-1) \gamma_i + \sum_{i=1}^{k} [u_i(p_n)] \gamma_{m+i} \right) \\
&= \frac{\phi (h!)}{h!} \sum_{j=1}^{h!} e\left( - \frac{j}{h!}\right) \frac{1}{W(N)} \sum_{n=1}^N w(n) \, e  \left(\sum_{i=1}^m  q_i ( p_n - 1) \gamma_i + \sum_{i=1}^{k} [u_i( p_n)] \gamma_{m+i} + \frac{p_n j}{h!}\right). 
\end{align*}

We will utilize again the partition $ \mathbb{T}^{m+k} = I \cup J_h \cup K_h$ defined in \eqref{decomposition:torus}. 
If $(\gamma_1, \dots, \gamma_{m+k}) \in I$, then $\lim\limits_{N \rightarrow \infty} G_N (\gamma_1, \dots, \gamma_{m+k}) = 0$ by Proposition \ref{prop:app:ud-ave}. 
If $(\gamma_1, \dots, \gamma_{m+k}) \in J_h$, then 
$$ \lim_{N \rightarrow \infty} G_N (\gamma_1, \dots, \gamma_{m+k}) = \lim_{N \rightarrow \infty} \frac{\phi(h!)}{W(N)} \sum_{n=1}^N w(n) \, 1_{h!\mathbb{Z} +1} (p_n) = \lim_{N \rightarrow \infty} \frac{\phi(h!)}{N} \sum_{n=1}^N 1_{h!\mathbb{Z} +1} (p_n)  = 1.$$
(The second equality follows from Theorem \ref{lem:Tsu1}, and the last one follows from the prime number theorem along arithmetic progressions.)
If $(\gamma_1, \dots, \gamma_{m+k}) \in K_h$, then 
$ |G_N (\gamma_1, \dots, \gamma_{m+k})| \leq  \frac{\phi(h!)}{W(N)} \sum_{n=1}^N w(n) \, 1_{h!\mathbb{Z} +1} (p_n)$, so
$\limsup\limits_{N \rightarrow \infty} |G_N (\gamma_1, \dots, \gamma_{m+k})|  \leq 1.$
Therefore, we obtain \eqref{thm:setofrec:prime}. 
 
The proof that $D_3$ is a set of recurrence is analogous to the proof for $D_2$ and therefore is omitted.
\end{proof}

Recall that the upper Banach density of a set $E \subset \mathbb{Z}^k$ is defined by
$$d^*(E) = \sup_{\{\Pi_n\}_{n \in \mathbb{N}}} \limsup_{n \rightarrow \infty} \frac{|E \cap \Pi_n|}{| \Pi_n|},$$
where the supremum is taken over all sequences of parallelepipeds
$$ \Pi_n = [a_n^{(1)}, b_n^{(1)}] \times \cdots \times [a_n^{(k)}, b_n^{(k)} ] \subset \mathbb{Z}^k, \,\, n \in \mathbb{N},$$
with $b_n^{(j)} - a_n^{(j)} \rightarrow \infty$ for each $1 \leq j \leq k$.

By the $\mathbb{Z}^k$-version of Furstenberg's correspondence principle (see, for example, Proposition 7.2 in \cite{BMc}), 
given $E \subset \mathbb{Z}^k$ with $d^*(E)>0$, there is a probability space $(X, \mathcal{B} , \mu)$, commuting invertible measure preserving transformations $T_1, T_2, \dots , T_k$ of $X$, and $A \in \mathcal{B} $ with $d^*(E) = \mu(A)$ such that for any $\bold{n_1}, \bold{n_2}, \dots , \bold{n_m} \in \mathbb{Z}^k$, one has
$$d^*(E \cap (E- \bold{n_1}) \cap (E- \bold{n_2}) \cap \cdots \cap (E - \bold{n_m})) \geq \mu(A \cap T^{- \bold{n_1}}A \cap \cdots \cap T^{-\bold{n_m}} A),$$
where for $\bold{n} = (n_1, \dots , n_k)$, $T^{\bold{n}} = T_1^{n_1} \cdots T_k^{n_k}.$

Theorem \ref{thm:setofrec:H} and Furstenberg's correspondence principle now imply the following result.
\begin{Corollary}
\label{cor:setofrec:H}
Assume that $q_1(x), \dots, q_m(x) \in \mathbb{Z}[x]$ satisfy $q_i(0) = 0$ for all $i = 1, 2, \dots, m$, and let  $u_1(x), \dots, u_k(x)$ be subpolynomial functions belonging to some Hardy field $\mathbf{H}$.
Suppose that there exists a function $W(x) \in \mathbf{E}$  such that 
\begin{enumerate}
\item $\lim\limits_{x \rightarrow \infty} W(x) = \infty$ and $w(x) := W'(x)$ is non-increasing and positive, 
\item for any $u(x) \in \text{span}_{\mathbb{R}}^* (u_1(x), \dots, u_k(x))$ and any $q(x) \in \mathbb{R}[x]$, 
\begin{equation}
\lim_{x \rightarrow \infty} \frac{|u(x) - q(x)|}{\log W(x)} = \infty.
\end{equation}
\end{enumerate}
For any set $E \subset \mathbb{Z}^{m+d}$ with $d^*(E) > 0$,
\begin{align*}
 (E- E) \cap   \{  (q_1(n), \dots, q_m(n), \lfloor u_1(n) \rfloor, \dots, \lfloor u_k(n) \rfloor ) : n \in \mathbb{N} \} &\ne \emptyset,    \\
 (E- E) \cap  \{  (q_1(p-1), \dots, q_m(p-1), \lfloor u_1(p) \rfloor, \dots, \lfloor u_k(p) \rfloor ) : p \in \mathcal{P} \} &\ne \emptyset,    \\
 (E- E) \cap  \{  (q_1(p+1), \dots, q_m(p+1), \lfloor u_1(p) \rfloor, \dots, \lfloor u_k(p) \rfloor ) : p \in \mathcal{P} \} &\ne \emptyset.  
\end{align*}
\end{Corollary}

Let $q_1(x), \dots, q_m(x) \in \mathbb{Z}[x]$ be polynomials with $q_i(0) = 0$ for all $i = 1, 2, \dots, m$, and let $u_1(x), \dots, u_k(x)$ be subpolynomial functions belonging to some Hardy field $\mathbf{H}$. 
Suppose that there exists a function $W(x) \in \mathbf{E}$ such that 
\begin{enumerate}[(i)]
\item $\lim\limits_{x \rightarrow \infty} W(x) = \infty$ and $w(x) := W'(x)$ is non-increasing and positive; 
\item for any $u(x) \in \text{span}_{\mathbb{R}}^* (u_1(x), \dots, u_k(x))$ and any $q(x) \in \mathbb{R}[x]$, 
\begin{equation}
\label{cond:eq:thm:setofrec:H}
\lim_{x \rightarrow \infty} \frac{|u(x) - q(x)|}{\log W(x)} = \infty.
\end{equation}
\end{enumerate}

The following theorem extends Theorem \ref{thm:setofrec:H} to sets of recurrence involving more general expressions.  
\begin{Theorem}
\label{thm:rec1}
Let $q_1(x), \dots, q_m(x) \in \mathbb{Z}[x]$ be polynomials with $q_i(0) = 0$ for all $i = 1, 2, \dots, m$.  
Let $u_1(x), \dots, u_k(x)$ be subpolynomial functions belonging to some Hardy field $\mathbf{H}$ and let $v_1(x), \dots, v_l(x)$ be subpolynomial Hardy field functions such that $\lim\limits_{x \rightarrow \infty} \frac{|v_j(x)|}{\log x} < \infty$ for $j = 1, 2, \dots, l$. 
\begin{enumerate}
\item Suppose that  $v_1(x \log x), \dots, v_l(x \log x) \in {\mathbf H}$ and there exists a function $W(x) \in \mathbf{E}$ such that 
\begin{enumerate}[(i)]
\item $\lim\limits_{x \rightarrow \infty} W(x) = \infty$ and $w(x) := W'(x)$ is non-increasing and positive; 
\item for any $u(x) \in \text{span}_{\mathbb{R}}^* \{ u_1(x), \dots, u_k(x), v_1(x \log x), \dots, v_l(x \log x) \}$,
\[ \lim_{x \rightarrow \infty} \frac{|u(x) -q(x)|}{\log W(x)} = \infty \quad \text{ for any } q(x) \in \mathbb{Q}[x].\] 
\end{enumerate}
Then 
\begin{equation}
\label{rec:D1}
D_1 = \{ (q_1 (n), \dots, q_m(n), \lfloor u_1(n) \rfloor, \dots, \lfloor u_k (n) \rfloor, \lfloor v_1(p_n) \rfloor, \dots, \lfloor v_l(p_n) \rfloor) : n \in \mathbb{N}   \}
\end{equation}
is a set of recurrence. 
\item Suppose that $v_1 \circ g (x), \dots, v_l \circ g (x) \in {\mathbf H}$ and there exists a function $W(x) \in \mathbf{E}$ such that 
\begin{enumerate}[(i)]
\item $\lim\limits_{x \rightarrow \infty} W(x) = \infty$ and $w(x) := W'(x)$ is non-increasing and positive; 
\item for any ${u(x) \in \text{span}_{\mathbb{R}}^* \{ u_1(x), \dots, u_k(x), v_1 \circ g (x), \dots, v_l \circ g (x) \}}$,
\[ \lim_{x \rightarrow \infty} \frac{|u(x) -q(x)|}{\log W(x)} = \infty \quad \text{ for any } q(x) \in \mathbb{Q}[x].\] 
\end{enumerate}
Then 
\begin{equation}
\label{rec:D2}
D_2 = \{ (q_1 (p_n - 1), \dots, q_m(p_n - 1), \lfloor u_1(p_n) \rfloor, \dots, \lfloor u_k (p_n) \rfloor, \lfloor v_1(n) \rfloor, \dots, \lfloor v_l(n) \rfloor) : n \in \mathbb{N}   \}
\end{equation}
and 
\begin{equation}
\label{rec:D3}
D_3 = \{ (q_1 (p_n + 1), \dots, q_m(p_n + 1), \lfloor u_1(p_n) \rfloor, \dots, \lfloor u_k (p_n) \rfloor, \lfloor v_1(n) \rfloor, \dots, \lfloor v_l(n) \rfloor) : n \in \mathbb{N}   \}
\end{equation}
are sets of recurrence. 
\end{enumerate}
\end{Theorem}

\begin{proof} Let us prove that $D_1$ is a set of recurrence. (The proofs that $D_2$ and $D_3$ are sets of recurrence are analogous and are omitted.) 
Note that, as in the proof of Theorem \ref{thm:setofrec:H}], one can find $h \in \mathbb{N}$ such that 
\begin{equation}
\label{eq:2.19}
\lim_{N \rightarrow \infty} \frac{1}{W(N)} \sum_{n=1}^N w(n) 1_{h! \mathbb{Z}} (n) \mu (A \cap T_1^{-q_1( n)} \cdots T_m^{-q_m( n)} T_{m+1}^{- \lfloor u_1( n) \rfloor } \cdots T_{m+k}^{- \lfloor u_k (n) \rfloor } T_{m+k+1}^{- \lfloor v_1(n \log n) \rfloor} \cdots T_{m+k+l}^{- \lfloor v_l( n \log n) \rfloor} A) > 0.
\end{equation}
By Lemma \ref{lem:comp:slow:prime}, we can replace $v_i (n \log n)$ with $v_i (p_{n})$ in \eqref{eq:2.19}, which implies that $D_1$ is a set of recurrence.  
\end{proof}

The following result is a consequence of Theorem \ref{thm:rec1} and Furstenberg's correspondence principle. 
\begin{Corollary}
\label{cor:combi-rec}
Let $D_1, D_2, D_3$ be the sets defined in \eqref{rec:D1}, \eqref{rec:D2}, \eqref{rec:D3}.
Then, for any set $E \subset \mathbb{Z}^{m+k+l}$ with $d^*(E) > 0$,
\[ (E-E) \cap D_i \ne \emptyset \quad \text{ for } i = 1, 2, 3.\]
\end{Corollary}

In particular, this leads to a new application to sets of recurrence:  for any $\alpha \in \mathbb{N}$, $\beta \in \mathbb{R}$ with $\beta > 0$ and  $\beta \notin \mathbb{N}$, $\gamma >0$, and $\delta \in (0, 1]$, if $E \subset \mathbb{Z}^4$ with $d^* (E) > 0$, then there exists $n \in \mathbb{N}$ such that
\begin{equation} 
(n^{\alpha}, \lfloor n^{\beta} \rfloor, \lfloor \log^{\gamma} n \rfloor, \lfloor \log^{\delta} p_n \rfloor) \in E - E. 
\end{equation}

\subsection{A multiple recurrence result and an application}
\label{subsec:3.3} 

In this subsection, we prove Theorem \ref{thm:multiple:slow}, which is a multiple recurrence result along sequences of the form \( q_j(n) + [v_j(p_n)] \), where \( q_j \) are polynomials and \( v_j \) are slowly growing functions from a Hardy field.
We also derive a combinatorial application of this result. 
One of the main ingredients in the proof is Theorem B from \cite{BMR}, which we now state.
Before doing so, we introduce the necessary notation and definitions.

Given functions $u_1(x), \dots, u_k(x)$ in a Hardy field $\mathbf{H}$, define 
\[ \text{poly } (u_1, \dots, u_k) = \{p(x) \in \mathbb{R}[x] : \exists u \in \text{span} (u_1, \dots, u_k) \text{ with } \lim\limits_{x \rightarrow \infty} |u(x) - p(x)| =0 \}\]
and
\[ \nabla\text{-span}(u_1, \dots, u_k) := \{ c_1 u_1^{(m_1)} (x) + \cdots + c_k u_k^{(m_k)} (x) : c_1, \dots, c_k \in \mathbb{R}, m_1, \dots, m_k \in \mathbb{N} \cup \{0\}  \}, \]
where $u_j^{(m_j)} (x)$ denotes the $m_j$-th derivative of $u_j(x)$.

\begin{Definition} Let $\mathbf{H}$ be a Hardy field. Let $W(x) \in {\mathbf E}$ be such that $\lim\limits_{x \rightarrow \infty} W(x) = \infty$ and $w(x) := W'(x)$ is non-increasing and positive. 
We say that a weight function $W(x)$ and functions $u_1(x), \dots, u_k(x) \in \mathbf{H}$ satisfy {\em property (P)} if for all $u \in  \nabla \text{-span} (u_1, \dots, u_k)$ and $p(x) \in \mathbb{R}[x]$, either $|u(x) - p(x)| \ll 1$ or $|u(x) - p(x)| \succ \log W(x)$.
\end{Definition}

\begin{Theorem}[Theorem B in \cite{BMR}]
\label{thm:summary:bmr}  
Let $u_1(x), \dots, u_k(x)$ be subpolynomial functions in a Hardy field $\mathbf H$.
Assume that $W(x)$ and $u_1(x), \dots, u_k(x)$ satisfy property (P). 
Let $(X, \mathcal{B}, \mu, T)$ be an invertible measure preserving system.
\begin{enumerate}
\item For any $f_1, \dots, f_k \in L^{\infty}(X)$, the limit 
\begin{equation}
\lim_{N \rightarrow \infty} \frac{1}{W(N)} \sum_{n=1}^N w(n) T^{[ u_1(n) ]} f_1 \cdots T^{[ u_k(n) ]} f_k
\end{equation}
exists in $L^2$.
\item Suppose that for any $q(x) \in \mathbb{Z}[x]$ and $u \in \text{span}^* (u_1, \dots, u_k)$, $\lim\limits_{x \rightarrow \infty} |u(x) -q(x)| = \infty$. 
Then, for any $A \in \mathcal{B}$, 
\begin{equation}
\lim_{N \rightarrow \infty} \frac{1}{W(N)} \sum_{n=1}^N w(n)  \mu(A \cap T^{- [ u_1(n) ]} A \cap \cdots \cap T^{- [ u_k(n) ]} A) \geq \mu (A)^{k+1}
\end{equation}
\item Suppose that there is a jointly intersective\footnote{A finite collection of polynomials $q_1(x), \dots, q_l(x) \in \mathbb{Z}[x]$ is called {\em jointly intersective} if for all $m \in \mathbb{N}$, there is $n \in \mathbb{N}$ such that $q_1(n), \dots, q_l (n)$ are divisible by $m$.} collection of polynomials $q_1(x), \dots, q_l(x) \in \mathbb{Z}[x]$ such that 
\[ \text{poly} \, (u_1, \dots, u_k) \subset \text{ span} \, (q_1, \dots, q_l).\]
Then for any $A \in \mathcal{B}$ with $\mu(A) > 0$, we have 
\begin{equation}
\lim_{N \rightarrow \infty} \frac{1}{W(N)} \sum_{n=1}^N w(n) \mu (A \cap T^{- [ u_1(n) ]} A \cap \cdots \cap T^{- [ u_k(n) ]} A) > 0
\end{equation}
\end{enumerate}
\end{Theorem}

\begin{Remark}
In statements (1) and (2) of Theorem \ref{thm:summary:bmr}, the rounding-to-the-nearest-integer function $[\cdot]$ can be replaced by either the floor function $\lfloor \cdot \rfloor$ or the ceiling function $\lceil \cdot \rceil$. However, this substitution does not apply to statement (3); see Remarks 1.8 and 1.9 in \cite{BMR}.
\end{Remark}

By combining Lemma \ref{lem:comp:slow:prime} with Theorem \ref{thm:summary:bmr}, we obtain the following result concerning multiple ergodic averages along sequences of the form \( q_j(n) + [v_j(p_n)] \), where \( q_j \) are polynomials and \( v_j \) are slowly growing Hardy field functions.

\begin{Theorem}
\label{thm:multiple:slow}
Let $q_1(x), \dots, q_k(x) \in \mathbb{Z}[x]$ and $v_1(x), \dots, v_k(x)$ be functions in a Hardy field such that for each $j = 1, 2, \dots, k$, 
$\lim\limits_{x \rightarrow \infty} \frac{|v_j(x)|}{\log x} < \infty$. 
Suppose that 
\begin{equation}
\label{eq:cond:thm;multi}
\lim_{x \rightarrow \infty} \frac{| c_1 v_1 (x) + \cdots + c_k v_k(x)|}{\log W(x)} = \infty \quad \text{ for any } (c_1, \dots, c_k) \in \mathbb{R}^k \setminus \{ (0, 0, \dots, 0) \}
\end{equation}
\begin{enumerate}
\item For any  measure preserving system $(X, \mathcal{B}, \mu, T)$, for any $f_1, \dots, f_k, g_1, \dots, g_k \in L^{\infty}$, 
\begin{equation}
\label{eq1:thm:multi}
\lim_{N \rightarrow \infty} \frac{1}{W(N)} \sum_{n=1}^N w(n) \prod_{j=1}^k T^{q_j(n)} f_j \cdot T^{ q_j(n) + [ v_j (p_n) ]  } g_j 
\end{equation}
exists in $L^2$.
\item For any  measure preserving system $(X, \mathcal{B}, \mu, T)$ and for any $A \in \mathcal{B}$, 
\begin{equation}
\label{eq1:thm:multi2}
\lim_{N \rightarrow \infty} \frac{1}{W(N)} \sum_{n=1}^N w(n)  \mu(A \cap T^{- q_1(n) - [ v_1(p_n) ] } A \cap \cdots \cap T^{- q_k(n) - [ v_k(p_n) ] } A) \geq \mu (A)^{k+1}.
\end{equation}
\item Suppose that polynomials $q_1(x), \dots, q_k(x)$ are jointly intersective. Then, for any  measure preserving system $(X, \mathcal{B}, \mu, T)$ and for any $A \in \mathcal{B}$, 
\begin{equation}
\label{eq1:thm:multi3}
\lim_{N \rightarrow \infty} \frac{1}{W(N)} \sum_{n=1}^N w(n)  \mu(A \cap T^{- q_1(n) - [ v_1(p_n) ]  } A \cap \cdots \cap T^{- q_k(n) - [ v_k(p_n) ]  } A) > 0.
\end{equation}
\end{enumerate}
\end{Theorem}

\begin{proof}
Let $D =  \bigcup\limits_{j=1}^k D_j$, where $D_j = \{n: [v_j(n \log n)] \ne [v_j(p_n)] \}$ 
Since $[x] = \lfloor x +1/2 \rfloor$, we can write $$D= \{n: \lfloor v_j(n \log n) +1/2 \rfloor \ne \lfloor v_j(p_n) +1/2 \rfloor \}.$$ 
By Lemma \ref{lem:comp:slow:prime}, 
\[ \lim_{N \rightarrow \infty} \frac{1}{W(N)} \sum_{n=1}^N w(n) 1_D(n) = 0.\]  
Thus, it is enough to show that formulas \eqref{eq1:thm:multi} - \eqref{eq1:thm:multi3} hold when $ \lfloor v_j (p_n)  \rfloor $ is replaced by $\lfloor v_j (n \log n)  \rfloor$.

Since polynomials belong to every maximal Hardy field, there exists a Hardy field $\mathbf{H}$ such that 
$$q_1(x), \dots, q_k(x), v_1(x \log x), \dots, v_k(x \log x) \in \mathbf{H}.$$
Also, the condition $\lim\limits_{x \rightarrow \infty} \frac{|v_j(x)|}{\log x} < \infty$ implies that
$$\lim\limits_{x \rightarrow \infty} v_j' (x) = \lim_{x \rightarrow \infty} \frac{v_j(x)}{x} = 0$$ 
for all $j = 1, \dots, k $.
Thus, from condition \eqref{eq:cond:thm;multi}, we conclude that
\begin{enumerate}[(i)] 
\item $q_1(x) + v_1(x \log x), \dots, q_k(x) + v_k (x \log x)$ and $W(x)$ satisfy property $(P)$,
\item for any $q(x) \in \mathbb{Z}[x]$ and $v(x) \in \text{span}^* (q_1(x)+ v_1(x \log x), \dots, q_k(x)+v_k (x \log x))$, 
$$\lim_{x \rightarrow \infty} |v(x) - q(x)| = \infty.$$
\end{enumerate}
Therefore, by statements (1) and (2) of Theorem \ref{thm:summary:bmr}, \eqref{eq1:thm:multi} and \eqref{eq1:thm:multi2} hold when $ \lfloor v_j (p_n)  \rfloor $ is replaced with $\lfloor v_j (n \log n)  \rfloor$. 

Now suppose that $q_1(x), \dots, q_k(x)$ are jointly intersective. 
Condition \eqref{eq:cond:thm;multi} implies
\[ \text{poly } (q_1, \dots, q_k, q_1+v_1, \dots, q_k+v_k) \subset \text{ span } (q_1, \dots, q_k).\]
Hence, by statement (3) of Theorem  \ref{thm:summary:bmr}, \eqref{eq1:thm:multi3} holds when $ \lfloor v_j (p_n)  \rfloor $ is replaced with $\lfloor v_j (n \log n)  \rfloor$.
\end{proof}

The following combinatorial corollary now follows from Furstenberg's correspondence principle and items (2) and (3) of Theorem \ref{thm:multiple:slow}.
\begin{Corollary}
\label{Cor:multrec}
Let $q_1(x), \dots, q_k(x) \in \mathbb{Z}[x]$.  
Let $v_1(x), \dots, v_k(x)$ be functions belonging to a Hardy field such that for each $j = 1, 2, \dots, k$, 
$\lim\limits_{x \rightarrow \infty} \frac{|v_j(x)|}{\log x} < \infty$.
Suppose that there exists $W(x) \in {\mathbf E}$ with the property that $\lim\limits_{x \rightarrow \infty} W(x) = \infty$, $w(x) := W'(x)$ is non-increasing and positive and for any $(c_1, \dots, c_k) \in \mathbb{R}^k \setminus \{ (0, 0, \dots, 0) \}$, 
\begin{equation*}
\lim_{x \rightarrow \infty} \frac{| c_1 v_1 (x) + \cdots + c_k v_k(x)|}{\log W(x)} = \infty.
\end{equation*}
Then, for any $E \subset \mathbb{N}$ with $d^* (E) > 0$, there exist $a$ and $n \in \mathbb{N}$ such that 
\begin{equation}
\label{eq1:corLmultrecu}
a,  a + q_1 (n) + [ v_1 (p_n) ], \cdots,  a+ q_k(n) + [ v_k(p_n) ] \in E.
\end{equation}
In addition, assume that the polynomials $q_1(x), \dots, q_k(x)$ are jointly intersective. 
Then, for any $E \subset \mathbb{N}$ with $d^* (E) > 0$, there exist $a$ and $n \in \mathbb{N}$ such that 
\begin{equation}
\label{eq2:corLmultrecu}
a, a + q_1(n), a + q_1 (n) + [ v_1 (p_n) ], \cdots, a+ q_k(n), a+ q_k(n) + [ v_k(p_n) ] \in E.
\end{equation}
\end{Corollary}

\begin{Remark}
If we define $w_j(x) = v_j(x) + 1/2$ for $1 \leq j \leq k$, then $w_1, \dots, w_k$ also satisfy all the assumptions of Corollary \ref{Cor:multrec}. Since $[x] = \lfloor x+ 1/2 \rfloor$, one may replace $[v_j (p_n)]$ in \eqref{eq1:corLmultrecu} and \eqref{eq2:corLmultrecu} by $\lfloor v_j(p_n) \rfloor$.
\end{Remark}

\subsection{Benford law}\label{subsec:3.4} 

The first digit problem concerns the relative frequency distribution of leading digits in datasets.
Let $(x_n)_{n \in \mathbb{N}}$ be a sequence of positive real numbers. 
We say that $(x_n)_{n \in \mathbb{N}}$ satisfies the {\em first digit law} if, for any $a \in \{1, 2, \dots, 9\}$, 
\begin{equation}
\lim_{N \rightarrow \infty} \frac{1}{N} \left| \{1 \leq n \leq N:  \text{the first digit of } x_n \text{ is } a\} \right| = \log_{10} (1+1/a).
\end{equation} 
A classical example of a sequence satisfying this law is  $(2^n)_{n \in \mathbb{N}}$ (see, for example, \cite{AA}). 
Note that 
the first digit of $2^n$ is $a$ if and only if $\{ \log_{10} 2^n \} \in [\log_{10} a, \log_{10}(a+1)).$ 
Therefore, the fact that the sequence $(2^n)_{n \in \mathbb{N}}$ satisfies the first digit law follows from the uniform distribution modulo $1$ of the sequence $(\log_{10} 2^n)_{n \in \mathbb{N}}$.

One can also consider the frequency distribution of digits beyond the first one.
We say that a sequence $(x_n)_{n \in \mathbb{N}}$ of positive real numbers is {\em (strong) Benford} if $(\log_{10} x_n)_{n \in \mathbb{N}}$ is uniformly distributed $\bmod \, 1$.  
It is known (see Theorem 1 in \cite{Dia}) that this condition is equivalent to the following: 
for any finite string $a_1 a_2 \dots a_k$ of digits, where $k\in \mathbb{N}$, $a_1 \in \{1, 2, \dots, 9\}$, and $a_i \in \{0, 1, 2, \dots, 9\}$ for $2 \leq i \leq k$, the frequency with which $x_n$ starts with $a_1 a_2 \dots a_k$ is given by $$\log_{10} (1+ 1/S),$$
where $S = \sum\limits_{j=1}^k a_j 10^{k-j}$.
In the remainder of this section, we omit the adjective ``strong" when referring to Benford sequences.

We now list some examples of sequences that satisfy the Benford law, which can be verified using Theorem \ref{BKS:prev}.
\begin{enumerate}
\item The sequences $(n^n)_{n \in \mathbb{N}}$ and $(p_n^{p_n})_{n \in \mathbb{N}}$ are Benford. Indeed, by Theorem \ref{BKS:prev}, the sequences $(n \log_{10} n)_{n \in \mathbb{N}}$ and  $(p_n \log_{10} p_n)_{n \in \mathbb{N}}$ are uniformly distributed $\bmod \, 1$.
\item The sequences $(n!)_{n \in \mathbb{N}}$ and $(p_n!)_{n \in \mathbb{N}}$ are Benford. This follows from the fact that $\log_{10} \Gamma(x)$ belongs to a Hardy field, and Theorem \ref{BKS:prev} implies that $(\log_{10} \Gamma(n))_{n \in \mathbb{N}}$ and $(\log_{10} \Gamma(p_n))_{n \in \mathbb{N}}$ are uniformly distributed $\bmod \, 1$. (See also Theorem 3 in \cite{Dia} for a proof that $(n!)_{n \in \mathbb{N}}$ is a Benford sequence.)
\end{enumerate}

It is well known that neither the sequence of natural numbers nor that of prime numbers satisfies the first digit law (see, for example, \cite{MaSch11} and \cite{Wi}). 
On the other hand, Duncan \cite{Dun} showed that the logarithmic density of the set $A_a = \{n \in \mathbb{N}: \text{the first digit of } n \text{ is } a\}$ is $\log_{10} (1 + 1/a)$:
\begin{equation} 
\label{eq:4.58}
\lim_{N \rightarrow \infty} \frac{1}{\log N} \sum_{n=1}^N \frac{1}{n} 1_{A_a}(n) = \log_{10} (1 + 1/a).
\end{equation}
Also, in \cite{Wh}, Whitney proved that the sequence of primes $(p_n)_{n \in \mathbb{N}}$ satisfies the logarithmic density law:
\begin{equation} 
\label{eq:4.59}
\lim_{N \rightarrow \infty} \frac{1}{\log N} \sum_{n=1}^N \frac{1}{n} 1_{A_a}(p_n) = \log_{10} (1 + 1/a).
\end{equation}
Note that \eqref{eq:4.58} and \eqref{eq:4.59} follow from the fact that $(\log_{10} n)_{n \in \mathbb{N}}$ and $(\log_{10} p_n)_{n \in \mathbb{N}}$ are $(1/n)$-uniformly distributed $\bmod \,1$. (See Theorem \ref{BKS}.)

Borrowing the terminology from \cite{MaSch}, we will say that a sequence $(x_n)_{n \in \mathbb{N}}$ is {\em log-Benford} if $(\log_{10} x_n)_{n \in \mathbb{N}}$ is $(1/n)$-uniformly distributed $\bmod \, 1$.
In a similar fashion, we will say that a sequence $(x_n)_{n \in \mathbb{N}}$ is loglog-Benford if $(\log_{10} x_n)_{n \in \mathbb{N}}$ is $\left( \frac{1}{n \log n} \right)$-uniformly distributed $\bmod \, 1$.
  
\begin{Theorem}
\label{thm:ex:simpleprod}
Let 
\[x_n = n^{t_1} p_n^{t_2} \, (n \in \mathbb{N}),\]
where $t_1, t_2 \in \mathbb{R}$ with $(t_1, t_2) \ne (0,0)$.  
\begin{enumerate}
\item If $t_1 + t_2 \ne 0$, then $(x_n)_{n \in \mathbb{N}}$ is not Benford, but it is log-Benford.
\item If $t_1 + t_2 = 0$, then $(x_n)_{n \in \mathbb{N}}$ is not log-Benford, but it is loglog-Benford.
\end{enumerate}
\end{Theorem}
\begin{proof}
By Theorem \ref{BKS}, we have that
\begin{enumerate}
\item if $t_1 + t_2 \ne 0$, then the sequence $(t_1 \log_{10} n + t_2 \log_{10} (n \log n))_{n \in \mathbb{N}}$ is not uniformly distributed, but $(1/n)$-uniformly distributed 
\item  if $t_1 + t_2 = 0$, then  the sequence $(t_1 \log_{10} n + t_2 \log_{10} (n \log n))_{n \in \mathbb{N}}$ is not $(1/n)$-uniformly distributed, but  $(1/n \log n)$-uniformly distributed.
\end{enumerate}
On the other hand, Theorem \ref{lem2:sec2} implies that the sequence $\log_{10} x_n = t_1 \log_{10} n + t_2 \log_{10} p_n$ is $w(n)$-uniformly distributed if and only if the sequence $(t_1 \log_{10} n + t_2 \log_{10} (n \log n))_{n \in \mathbb{N}}$ is $w(n)$-uniformly distributed.
Therefore, we have the following:
\begin{enumerate}
\item If $t_1 + t_2 \ne 0$, then the sequence $(\log_{10} x_n)_{n \in \mathbb{N}}$ is not uniformly distributed, but $(1/n)$-uniformly distributed. It follows that $(x_n)_{n \in \mathbb{N}}$ is not Benford, but is log-Benford.
\item  If $t_1 + t_2 = 0$, then  the sequence $(\log_{10} x_n)_{n \in \mathbb{N}}$ is not $(1/n)$-uniformly distributed, but  $(1/n \log n)$-uniformly distributed. Thus, $(x_n)_{n \in \mathbb{N}}$ is not log-Benford, but is loglog-Benford.
\end{enumerate}
\end{proof}

We conclude this section by obtaining amplifications of two results established in \cite{MaSch}.
We start with the sequence $( \log 2 \times \cdots \times \log n)_{n}$. 
It was proved in \cite{MaSch} that this sequence is log-Benford and the question was posed whether it is actually Benford (see \cite{MaSch}, Theorem 4.1 and the discussion in Section 5).
The affirmative answer to this question corresponds to the case $j=1$ in the following more general result. 
Recall that $\log^{(j)}$ denotes the $j$-fold composition of the logarithm.

\begin{Theorem}
\label{thm:sum-log}
Let $j \in \mathbb{N}$.  
Define
$$U_n^{(j)} = \prod_{k= A_j}^n \log^{(j)} k,$$
where  $A_j \in \mathbb{N}$ is chosen so that $\log^{(j)} x$ is defined on $x \geq A_j$. 
Then $(U_n^{(j)})_n$ is Benford. 
\end{Theorem}
Theorem \ref{thm:sum-log} follows directly from Theorem \ref{thm:sum:ud:new} in Section \ref{sec2}.

Consider now the sequence of primorials $(p_n \#)_{n \in \mathbb{N}}$, where $p_n \# = \prod\limits_{k=1}^n p_k$.
It was shown in \cite{MaSch} that the sequence $(p_n \#)_{n \in \mathbb{N}}$ is Benford.
We will examine the more general sequence 
\[ x_n = \prod_{k=1}^n k^{t_1} p_k^{t_2} = (n!)^{t_1} (p_n\#)^{t_2},\]
where $t_1, t_2 \in \mathbb{R}$ with $(t_1, t_2) \ne (0,0)$. 

\begin{Theorem}
\label{thm:prod:Benford}
Let 
\[ x_n = \prod_{k=1}^n k^{t_1} p_k^{t_2},\]
where $t_1, t_2 \in \mathbb{R}$ with $(t_1, t_2) \ne (0,0)$. 
\begin{enumerate}
\item If $t_1 + t_2 \ne 0$, then $(x_n)_{n \in \mathbb{N}}$ is Benford.
\item If $t_1 +  t_2 = 0$, then $(x_n)_{n \in \mathbb{N}}$ is log-Benford.  
\end{enumerate}
\end{Theorem} 

\begin{Remark}
We do not know whether $(x_n)_{n \in \mathbb{N}}$ is Benford or not when $t_1 + t_2 = 0$. 
\end{Remark}

\begin{proof}[Proof of Theorem \ref{thm:prod:Benford}.]
It suffices to prove that 
\begin{enumerate}[(i)]
\item for any $a, b \in \mathbb{R}$ with $a+b \ne 0$, 
\begin{equation}
\label{eq:cor:pairwise:log:thm3.31}
 \lim_{N \rightarrow \infty} \frac{1}{ N} \sum_{n=1}^N e ( a \log_{10} n! + b \log_{10} p_n \# ) = 0.
\end{equation} 
\item for any nonzero $a \in \mathbb{R}$,
\begin{equation}
\label{eq:cor:pairwise:log:thm3.31:2}
 \lim_{N \rightarrow \infty} \frac{1}{\log N} \sum_{n=1}^N \frac{1}{n}e ( a \log_{10} n!  - a \log_{10} p_n \# ) = 0.
\end{equation} 
\end{enumerate}

Let us prove \eqref{eq:cor:pairwise:log:thm3.31} first. 
Let $Q_n = \prod\limits_{m=2}^n m \log m$ for $n \geq 2$. 
Let $y_n = a \log_{10} n! + b \log_{10} p_n\#$.
Note that for any $n \geq 2$ and any $h \in \mathbb{N}$,
\[y_{n+h} - y_n = a \log_{10} ((n+h)!/n!) + b \log_{10} \frac{p_{n+h}\#/p_n\#}{Q_{n+h}/Q_n} + b \log_{10} (Q_{n+h}/Q_n).\]
We will use Theorem \ref{vdC} for $w(n)=1, n \in \mathbb{N}$.
By the triangle inequality, we have
\begin{equation}
\label{apply-vdC:thm3.31}
 \frac{1}{N} \left| \sum_{n=1}^N  e(y_{n+h} - y_n) \right| \leq \frac{1}{N} + \frac{2 \pi |b|}{ N} \sum_{n=2}^N  \log_{10} \frac{p_{n+h}\#/p_n\#}{Q_{n+h}/Q_n} + \frac{1}{ N} \left|\sum_{n=2}^N  e \left( b \log_{10} \frac{Q_{n+h}}{Q_n} + a \log_{10} \frac{(n+h)!}{n!} \right) \right|.
 \end{equation}

To estimate the quantity $\left( \frac{2 \pi |b|}{N} \sum\limits_{n=2}^N  \log_{10} \frac{p_{n+h}\#/p_n\#}{Q_{n+h}/Q_n} \right)$ on the right-hand side of \eqref{apply-vdC:thm3.31}, we will use the classical result of Roser (\cite{Ro}) which states that there exists $C > 0$ such that for all $n \in \mathbb{N}$,
 \begin{equation}
 \label{eq:primebound}
 n \log n \leq p_n \leq n \log n + C n \log \log n.
 \end{equation}
 It follows from \eqref{eq:primebound} that for $n \geq 3$,
 \[ 1 \leq  \frac{p_{n+h}\#/p_n\#}{Q_{n+h}/Q_n} \leq \prod_{j=1}^h \frac{(n+j) \log (n+j) + C(n+j) \log \log (n+j)}{(n+j) \log (n+j)} \leq \left( 1+ C \left( \frac{\log \log n}{\log n}\right) \right)^h,  \]
and so 
 \[ \frac{1}{ N} \sum_{n=3}^N  \log_{10} \frac{p_{n+h}\#/p_n\#}{Q_{n+h}/Q_n}  \leq \frac{Ch}{\log 10} \frac{1}{N} \sum_{n=3}^N \frac{\log \log n}{ \log n} \leq \frac{Ch}{\log 10} \frac{\log \log N}{\log N} \rightarrow 0 \]
as $N \rightarrow \infty.$ 

In order to apply Theorem \ref{vdC}, it remains to show that 
\begin{equation}
\label{thm:pari:eq:last:thm3.31}
\lim_{h \rightarrow \infty} \limsup_{N \rightarrow \infty} \frac{1}{ N} \left|\sum_{n=2}^N  e \left( b \log_{10} \frac{Q_{n+h}}{Q_n} + a \log_{10} \frac{(n+h)!}{n!} \right) \right| =0.
\end{equation}
Let $f(x) = (a+b) \sum\limits_{j=1}^h \log_{10} (x+j) + b \sum\limits_{j=1}^h \log_{10} \log (x+j)$.
Then, $f(n) = b \log_{10} (Q_{n+h}/Q_n) + a \log_{10} ((n+h)!/n!)$.
Note that
\[ \lim_{x \rightarrow \infty} |f(x) -  (a+b) h \log_{10} x -  b h \log_{10} \log x| = 0,\]
and so
\begin{equation} 
\label{eq:4.66}
\lim_{N \rightarrow \infty} \left| \frac{1}{N} \sum_{n=2}^N   \frac{1}{ N}  \sum_{n=2}^N  e \left( b \log_{10} \frac{Q_{n+h}}{Q_n} + a \log_{10} \frac{(n+h)!}{n!} \right) - \frac{1}{ N} \sum_{n=2}^N e \big( (a+b) h \log_{10} n  + bh \log_{10} \log n  \big)  \right| = 0.
\end{equation}
Applying Theorem \ref{lem:estimate:expsum:Eulersum} to $u(x) = (a+b) \log_{10} x  + b \log_{10} \log x$, we obtain
\begin{equation}
\label{eq:4.67}
 \lim_{h \rightarrow \infty} \limsup_{N \rightarrow \infty} \left| \frac{1}{ N} \sum_{n=2}^N e \big( (a+b) h \log_{10} n  + bh \log_{10} \log n  \big)  \right| = 0.
\end{equation}
Now \eqref{thm:pari:eq:last:thm3.31} follows from \eqref{eq:4.66} and \eqref{eq:4.67}. 
Thus, $\lim\limits_{h \rightarrow \infty} \limsup\limits_{N \rightarrow \infty} \frac{1}{N} \left| \sum\limits_{n=1}^N  e(y_{n+h} - y_n) \right| = 0$. Applying Theorem \ref{vdC}, we have \eqref{eq:cor:pairwise:log:thm3.31}.

Now let us prove \eqref{eq:cor:pairwise:log:thm3.31:2}: for any nonzero $a$,
\begin{equation}
\lim_{N \rightarrow \infty} \frac{1}{\log N} \sum_{n=1}^N \frac{1}{n}e ( a \log_{10} n!  - a \log_{10} p_n \# ) = 0.
\end{equation} 
Let $y_n  = a \log_{10} n! - a \log_{10} p_n\#$.
Note that for any $n \geq 2$ and any $h \in \mathbb{N}$,
\[y_{n+h} - y_n = - a \log_{10} \frac{p_{n+h}\#/p_n\#}{Q_{n+h}/Q_n} - a \log_{10} (Q_{n+h}/Q_n) + a \log_{10} ((n+h)!/n!).\]
We will use Theorem \ref{vdC} for $w(n)=1/n, n \in \mathbb{N}$. 
By the triangle inequality, we obtain
\begin{align}
\label{apply-vdC:2}
 \frac{1}{\log N} \left| \sum_{n=1}^N \frac{1}{n} e(y_{n+h} - y_n) \right| &\leq \frac{1}{\log N} + \frac{2 \pi |a|}{\log N} \sum_{n=2}^N \frac{1}{n} \log_{10} \frac{p_{n+h}\#/p_n\#}{Q_{n+h}/Q_n} \\
 &\quad + \frac{1}{\log N} \left|\sum_{n=2}^N \frac{1}{n} e \left( - a \log_{10} \frac{Q_{n+h}}{Q_n} + a \log_{10} \frac{(n+h)!}{n!} \right) \right|. \notag
 \end{align}
We use \eqref{eq:primebound} to estimate the quantity $\left( \frac{2 \pi |a|}{\log N} \sum\limits_{n=2}^N \frac{1}{n} \log_{10} \frac{p_{n+h}\#/p_n\#}{Q_{n+h}/Q_n} \right)$ on the right-hand side of \eqref{apply-vdC:2}: 
 \[ \frac{1}{\log N} \sum_{n=3}^N \frac{1}{n} \log_{10} \frac{p_{n+h}\#/p_n\#}{Q_{n+h}/Q_n}  \leq \frac{Ch}{\log N} \sum_{n=3}^N \frac{\log \log n}{n \log n} \leq Ch \frac{(\log \log N)^2}{\log N} \rightarrow 0 \]
as $N \rightarrow \infty.$ 

Thus, it remains to show that 
\begin{equation}
\label{thm:pari:eq:last}
\lim_{h \rightarrow \infty} \limsup_{N \rightarrow \infty} \frac{1}{\log N} \left|\sum_{n=2}^N \frac{1}{n} e \left( - a \log_{10} \frac{Q_{n+h}}{Q_n} + a \log_{10} \frac{(n+h)!}{n!} \right) \right| =0.
\end{equation}
Note that  $- a \log (Q_{n+h}/Q_n) + a \log ((n+h)!/n!) = - a \sum_{j=1}^h \log_{10} \log (n+j)$. 
Since $$\lim\limits_{x \rightarrow \infty} \left| \sum_{j=1}^h \log_{10} \log (x+j) - h \log_{10} \log x \right| =0,$$
we have that 
\begin{equation} 
\label{eq:comp:pair}
\lim_{N \rightarrow \infty} \left| \frac{1}{\log N} \sum_{n=2}^N \frac{1}{n} e \left(- a \sum_{j=1}^h \log_{10} \log (n+j) \right) 
- \frac{1}{\log N} \sum_{n=2}^n \frac{1}{n} e \left(- ah \log_{10} \log n\right)  \right| = 0.
\end{equation}
Applying Theorem \ref{lem:estimate:expsum:Eulersum} to $u(x) = -a \log_{10} \log x$, we have that 
\begin{equation}
\label{eq:apply:Euler}
\lim_{h \rightarrow \infty} \limsup_{N \rightarrow \infty} \frac{1}{\log N} \left| \sum_{n=2}^N \frac{1}{n} e \left(- ah \log_{10} \log n\right)  \right| =0. 
\end{equation}
Therefore, \eqref{thm:pari:eq:last} follows, and so does \eqref{eq:cor:pairwise:log:thm3.31:2}. 
 \end{proof} 

The following result is an immediate consequence of Theorem \ref{thm:prod:Benford} and Theorem \ref{WeylCriterion} (2).
\begin{Corollary}
\label{thm:primorial:factorial}
The sequence 
\[ (\log_{10} n!, \log_{10} p_n \#)_{n \in \mathbb{N}}\]
is $(1/n)$-uniformly distributed in $\mathbb{T}^2$. 
Therefore, $(n!, p_n \#)$ satisfies joint log-Benford law: for any strings of digits $S_1 = a_1 a_2 \dots a_l$ and $S_2 = b_1 b_2 \dots b_m$, where $l, m \in \mathbb{N}$, $a_1, b_1 \in \{1, 2, \dots, 9\}$ and $a_i, b_j \in \{0, 1, \dots, 9\}$ for $2 \leq i \leq l, 2 \leq j \leq m$, the logarithmic density of the set $A = \{n:  n! \text{ starts with } S_1 \text{ and } p_n\# \text{ starts with } S_2 \}$  is given by 
\[ \log_{10} (1+ 1/S_1) \cdot \log_{10} (1+ 1/S_2). \]
\end{Corollary}

\end{document}